\documentclass[10pt]{amsart}
\usepackage{amsfonts,amssymb,amscd,amsmath,enumerate,verbatim,calc}
\usepackage{mathrsfs}
\numberwithin{equation}{section}
\newtheorem{theorem}{Theorem}[section]
\newtheorem{corollary}[theorem]{Corollary}
\newtheorem{lemma}[theorem]{Lemma}
\newtheorem{proposition}[theorem]{Proposition}
\newtheorem{definition}[theorem]{Definition}

\newtheorem{example}[theorem]{Example}
\newtheorem{remark}[theorem]{Remark}

\begin{document}
% %----------------------------------------------------------------------
%%% ----------------------------------------------------------------------
%%% ----------------------------------------------------------------------

\title[  K\"{a}hler-Norden structures on Hom-Lie group and Hom-Lie algebras]
 { K\"{a}hler-Norden structures on\\ Hom-Lie groups and Hom-Lie algebras}

%%% ----------------------------------------------------------------------
%%% ----------------------------------------------------------------------
%%% ----------------------------------------------------------------------
\bibliographystyle{amsplain}
%%% ----------------------------------------------------------------------
%%% ----------------------------------------------------------------------
%%% ----------------------------------------------------------------------

 \author[E.  Peyghan, L. Nourmohammadifar and A.  Makhlouf]{ E. Peyghan, L. Nourmohammadifar, A. Makhlouf and A. Gezer}
 \address{E.P. and L.N. : Department of Mathematics, Faculty of Science, Arak University,
 	Arak, 38156-8-8349, Iran.\\
 	A.M. Universit\'e de Haute Alsace, IRIMAS-d\'epartement de math\'ematiques, Mulhouse, France.\\
 	A. G. Ataturk University, Faculty of Science, Department of Mathematics,
 	25240, Erzurum-Turkey.}
 \email{e-peyghan@araku.ac.ir,\ l.nourmohammadifar@gmail.com, abdenacer.makhlouf@uha.fr, agezer@atauni.edu.ir}

 %\author[M. Y. Sadeghi]{M. Y. Sadeghi$^{3}$}
%\address{$^{3}$ Department of Mathematics, Payame Noor University, Tehran,
 %19395-3697, Iran.}
%\email{m.sadeghi@phd.pnu.ac.ir}
%\thanks{ This research of the second author was in part supported by a grant from IPM (No. 93130022).}



%%% ----------------------------------------------------------------------
%%% ----------------------------------------------------------------------
%%% ----------------------------------------------------------------------

\begin{abstract}
 In the present paper, we describe  two geometric notions, holomorphic Norden structures and  K\"{a}hler-Norden structures on Hom-Lie groups,
and prove that on Hom-Lie groups in the left invariant setting, these structures are related to each other.  We study K\"{a}hler-Norden structures
with abelian complex structures and give the curvature properties
of  holomorphic Norden structures on Hom-Lie groups. Finally, we show that any left-invariant holomorphic Hom-Lie group is a flat (holomorphic Norden Hom-Lie algebra carries  a Hom-Left-symmetric algebra)  if its left-invariant complex structure (complex structure) is abelian. 
\end{abstract}

%%% ----------------------------------------------------------------------
%%% ----------------------------------------------------------------------
%%% ----------------------------------------------------------------------
\maketitle
%%% ----------------------------------------------------------------------
%%% ----------------------------------------------------------------------
%%% ----------------------------------------------------------------------

%%%%%%%%%%%%%%%%%%%%%%%%%%%%%%%%%%%%%%%%%%%%%%%%%%%%%%%%%%%%%%%%%%%%%%%%%%%%%%%%%%%%%%

%%%%%%%%%%%%%%%%%%%%%%%%%%%%%%%%%%%%%%%%%%%%%%%%%%%%%%%%%%%%%%%%%%%%%%%%%%%%%%%%%%%%%%

%%% ----------------------------------------------------------------------
%%% ----------------------------------------------------------------------
%%% ----------------------------------------------------------------------
%%% ----------------------------------------------------------------------
%%% ----------------------------------------------------------------------

%\newpage

%%%%%%%%%%%%%%%%%%%%%%%%%%%%%%%%%%%%%%%%%%%%%%%%%%%%%%%%%%%%

%%%%%%%%%%%%%%%%%%%%%%%%%%%%%%%%%%%%%5
%%%%%%%%%%%%%%%%%%%%%%%%%%%%%%%%%%%%%

\section{Introduction}

An almost complex structure on a $2n$-dimensional manifold $M$ is a field $J$ satisfying $J^2=-Id$.
A K\"{a}hler-Norden  manifold can be considered as
a triple $(M,g,J)$ which consists of a smooth manifold $M$ endowed with a Riemannian metric $g$ and  an almost complex structure $J$ such that $\nabla J = 0$ where $\nabla$ is the Levi-Civita connection of $g$ and the metric $g$ is assumed to be Anti-Hermitian(Norden), i.e., $g(JX,JY ) =- g(X,Y )$ for all vector fields $X$ and $Y$ on $M$. 
%\cite{GT}.
This kind of manifolds have been also studied
under the names: K\"{a}hlerian manifolds with Norden (or B) metrics,  anti-K\"{a}hler
manifolds (see \cite{GT, GM,SK}). The structure group of  K\"{a}hler-Norden manifolds is the complex
orthogonal group $O(n,\mathbb{C})$ and these manifolds have become an important area of research recently, mainly because of their applications in theoretical physics. For example 
 Andrzej Borowiec, Mauro Francaviglia, Marco Ferraris and Igor Volovich  proved that there is a one-to-one correspondence between (Einstein) holomorphic-Riemannian manifolds and (Einstein) anti-K\"{a}hlerian manifolds \cite{BFF1, BFF2}.
 The curvature properties of anti-K¨ahler manifolds have been studied in \cite{IS, O, S1}.
 Other strong motivation to study anti-K\"{a}hler manifolds comes from  
 the spinor geometry and geometric analysis on anti-K\"{a}hler manifolds and the generalized geometry of anti-Hermitian manifolds as complex Lie algebroids over anti-K\"{a}hler manifolds \cite{DN, DN1, N1}. 
When the manifold is a Lie group $G$, the metric and the complex structure are considered left-invariant, where they are both determined by their restrictions to the Lie algebra $\mathfrak{g}$ of $G$. In this situation, $(\mathfrak{g}, g_e,J_e)$ is called a  anti-K\"{a}hler Lie algebra. 
Anti-K\"{a}hler geometry on Lie groups
 have been studied by Edison Alberto  Fern\'{a}ndez-Culma and Yamile Godoy \cite{FG}.

The notion of Hom-Lie algebras was introduced by Hartwig, Larsson and Silvestrov in the study of   $\sigma$-deformations of the Witt and Virasoro algebras in 
\cite{HLS}. Indeed, some $q$-deformations of the
Witt and the Virasoro algebras have the structure of a Hom-Lie algebra \cite{HLS, H}. Based on the close relation between the discrete, deformed vector fields and differential calculus, this algebraic structure plays an important role in research fields \cite{HLS, LS, MS,MS1,MS2, PN}.
Moreover,  Recently, many scholars are very interested in the geometric problems over Hom-Lie algebras, such as infinite-dimensional Hom-Lie algebras, the
classical Hom-Yang-Baxter equation \cite{SB}, para-K\"{a}hler and  complex and K\"{a}hler structures on hom-Lie algebras \cite{PN, PN3},  complex product structures on Hom-Lie algebras and Hom-left symmetric algebroids \cite{PN4, PNA} and  Hom-Novikov algebras was introduced and studied by Yau in \cite{Y}.

%Hom-Lie algebras were widely studied. In particular, Hom-structures on
%semisimple Lie algebras were studied in [8, 20]; Geometrization of Hom-Lie algebras were studied in
%[3, 10]; Bialgebras for Hom-algebras were studied in [2, 16, 21, 22]; Recently, Elchinger, Lundengard,
%Makhlouf and Silvestrov extend the result in [7] to the case of (σ, τ)-derivations ([6]).

The concept of  Hom-groups was introduced as a non-associative analogue of a group in \cite{LM1}. 
 Then  in \cite{JMS}, J. Jiang, S. Mishra, and Y. Sheng  modified
the axioms in the  definition of Hom-group which is different from the one in \cite{LM1}.
 New definition leads to a Hom-Hopf algebra structure on the universal enveloping algebra of a Hom-Lie algebra and a Hom-group is associated  to any Hom-Lie algebra by considering group-like elements in its universal enveloping algebra.
 %The representations and a (co)homology theory of Hom-groups and Lagrange’s theorem for finite Hom-groups are  studied in \cite{H1, H2}.
 %The recent developments on Hom-groups make it natural to study Hom-Lie groups and to explore the relationship between Hom-Lie groups and Hom-Lie algebras. In \cite{JMS} 
 Also,  they  introduced the notion of a Hom-Lie group. In fact, a Hom-Lie group is considered as a Hom-group $(G,\diamond,e_\Phi,\Phi)$, where the underlying set $G$ is a (real) smooth manifold, the Hom-group operations (such as the product and the inverse) are smooth maps, and the underlying structure map $\Phi : G \rightarrow G$ is a diffeomorphism. Also a Hom-Lie algebra is associated to a Hom-Lie group by considering the notion of left-invariant sections of the pullback bundle $\Phi^!TG$. 

This work is intended as an attempt to study  two geometric notions, holomorphic Norden structures and  K\"{a}hler-Norden structures on Hom-Lie groups,
 to motivate on Hom-Lie groups in the left invariant setting, these structures are related to each other.  We study K\"{a}hler-Norden structures
 with abelian complex structures and give the curvature properties
 of  holomorphic Norden structures on Hom-Lie groups.

The paper is organized as follows.  In Section 2, we recall notions of 
Hom-algebra,  Hom-Lie algebra, Hom-group and Hom-Lie group.
%  and pseudo-Riemannian Hom-Lie algebras.
In Section 3,
we study  complex and Norden structures,  pseudo-Riemannian metric  and connections on Hom-Lie groups and Hom-Lie algebras.
% we introduce complex and Norden structures on Hom-Lie algebras.
 Also, we give examples of these structures.
   Then  we present   the theory of Tachibana operators  for Hom-Lie group and Hom-Lie algebras  and then 
give the notion of a holomorphic tensor on them.  In Section 4,  we introduce left-invariant K\"{a}hler-Norden structures on Hom-Lie groups and K\"{a}hler-Norden structures on Hom-Lie algebras and present examples of these structures.
We also  study  K\"{a}hler-Norden structures on Hom-Lie group and Hom-Lie algebras that their
 complex structure is abelian. Moreover, we  describe  Twin Norden metric of  Hom-Lie groups and Hom-Lie algebras. In  Section 5,
 we give the definition of a left-invariant holomorphic Norden Hom-Lie group (a holomorphic Norden Hom-Lie algebra).
 We show  that there exist a one-to-one correspondence between left invariant K\"{a}hler-Norden Hom-Lie groups (K\"{a}hler-Norden Hom-Lie algebras) and left-invariant holomorphic Norden Hom-Lie groups ( holomorphic Norden Hom-Lie algebras). 
  Finally, 
 we present some properties of  the Riemannian curvature tensor  of
 a left-invariant holomorphic Norden Hom-Lie group ( a holomorphic Norden Hom-Lie algebra) using the generalization of Tachibana operators, in  Section 6. Also, we show that any left-invariant holomorphic Hom-Lie group is a flat (holomorphic Norden Hom-Lie algebra carries  a Hom-Left-symmetric algebra )  if its left-invariant complex structure (complex structure) is abelian.

%\newpage
In this paper, we work over  the real field $\mathbb{R}$ and the complex field $\mathbb{C}$. In general case, it is denoted  by  $\mathbb{K}$.
%%%%%%%%%%%%%%%%%%%%%%%%%%%%%%%%%%%%%%%%%%%%%%%%%%%%%%%%%%%%
%\section{Preliminaries}

% % % % % % % % % % % % % % % % % % % % % % % % % % % % % % % % % % % % % % % % % % % % % % 
\section{Basic concepts on Hom-Lie groups}
% % % % % % % % % % % % % % % % % % % % % % % % % % % % % % % % % % % % % % % % % % % % % %

\begin{definition}\cite{PN}
A Hom-algebra is a triple $(V,\cdot , \phi_{V})$  consisting of a linear space $V$, a bilinear map (product)
$\cdot : V \times V \rightarrow V$ and an algebra morphism $\phi_{V} : V \rightarrow V$.
\end{definition}
\begin{definition}\cite{MS}
A Hom-Lie algebra is a triple $(\mathfrak{g}, [\cdot  , \cdot ]_\mathfrak{g},{ \phi_\mathfrak{g}})$ consisting of a linear space $\mathfrak{g}$, a
 bilinear map (bracket) $[\cdot , \cdot ]_\mathfrak{g}: \mathfrak{g}\times\mathfrak{g}\rightarrow\mathfrak{g}$ and an algebra morphism ${ \phi_\mathfrak{g}}:\mathfrak{g}\rightarrow \mathfrak{g}$ satisfying the following hom-Jacobi identity:
\begin{equation*}
[u,v]_\mathfrak{g}=-[v,u]_\mathfrak{g},\ \ \ \ \ \circlearrowleft_{u,v,w}[\phi_\mathfrak{g}(u),[v,w]_\mathfrak{g}]_\mathfrak{g}=0,
\end{equation*}
for any $u,v,w\in\mathfrak{g}$, where $\circlearrowleft$ is the symbol of cyclic sum.
 The second equation is called Hom-Jacobi identity.
\end{definition}
 The Hom-Lie algebra $(\mathfrak{g}, [\cdot  ,\cdot ]_\mathfrak{g}, { \phi_\mathfrak{g}})$ is called {\it regular Hom-Lie algebra} (respectively, {\it involutive Hom-Lie algebra}), if ${ \phi_\mathfrak{g}}$ is non-degenerate (respectively, satisfies ${ \phi_\mathfrak{g}}^2=Id_\mathfrak{g}$).
 It is known that a Lie algebra $(\mathfrak{g}, [\cdot, \cdot]_\mathfrak{g})$  with $\phi_\mathfrak{g}=Id_\mathfrak{g}$ is a Hom-Lie algebra. We call $(\mathfrak{g}, [\cdot, \cdot]_\mathfrak{g}, \phi_\mathfrak{g})$ \textit{proper Hom-Lie algebra} if $\phi_\mathfrak{g}\neq Id_\mathfrak{g}$.
 
%Also, a subspace $\mathfrak{h}\subset \mathfrak{g}$ is called a {\it Hom-Lie subalgebra} of $\mathfrak{g}$ if 
%$
%\phi_\mathfrak{g}(\mathfrak{h})\subset \mathfrak{h}$ and 
%$[u,v]\in\mathfrak{h}$, for any $ {u,v\in \mathfrak{h}}.
%$
%%%%%%%%%%%%%%%%%%%%%%%%%%%%%%%%%%%%%%%
\begin{definition}\cite{JMS}
	A  Hom-group is a set $G$ equipped with a product $\diamond:G\times G\rightarrow G$, a bijective map $\Phi: G \rightarrow G$ such that the following axioms are satisfied:\\
	(i)\ $
	\Phi(a\diamond b)=\Phi(a)\diamond\Phi(b),
$\\
	(ii)\  the product is Hom-associative, i.e.,
	$
	\Phi(a) \diamond  (b \diamond  c) = (a \diamond  b) \diamond  \Phi(c), 
$\\
	(iii) there exists a unique Hom-unit $e_\Phi\in G$ such that
	$a \diamond  e_\Phi = e_\Phi \diamond  a = \Phi(a)$, 
\\
	(iv) for each $a\in G$ there exists an element $a^{-1}$ satisfying the  condition
$	a \diamond   a^{-1} = a^{-1} \diamond   a = e_\Phi$,\\
for any $a,b,c \in G$.
	We denote a Hom-group by $(G,\diamond,e_\Phi,\Phi)$.
	\end{definition}
\begin{definition}\cite{JMS}
	A real Hom-Lie group is a Hom-group $(G,\diamond,e_\Phi,\Phi)$, in which $G$ is also a smooth real manifold, the map $\Phi:G\rightarrow G$ is a diffeomorphism, and the Hom-group operations (product and inversion) are smooth maps.
	\end{definition}
Assume  $(G,\diamond,e_\Phi,\Phi)$ is a Hom-Lie group.  The pullback map $\Phi^*:C^{\infty}(G)\rightarrow C^{\infty}(G)$ is a morphism of the function ring $C^{\infty}(G)$, i.e.,
\[
\Phi^*(fg)=\Phi^*(f)\Phi^*(g), \ \ \ \forall f,g\in  C^{\infty}(G).
\]
Let $A\rightarrow G$ be
a vector bundle of rank $n$. Denote by $\Gamma(A)$ the $C^\infty(G)$-module of sections of $A\rightarrow G$. A \textit{hom-bundle} is a triple $(A\rightarrow G, \Phi,\phi_A)$ consisting of a vector bundle $A\rightarrow G$, the smooth map $\Phi: G \rightarrow G$ and an algebra morphism $\phi_A:\Gamma(A)\rightarrow  \Gamma(A)$ satisfying
\[
\phi_A(fx)=\Phi^*(f)\phi_A(x),
\]
for any $x\in \Gamma(A)$ and $f\in C^{\infty}(G)$ (in this case, $\phi_A$ is called a $\Phi^*$-function linear).
% If $\varphi$ is a diffeomorphism and $\phi_A$ is an invertible map, then the hom-bundle $(A\rightarrow M, \varphi,\phi_A)$ is called invertible. 
Consider the tangent bundle $TG$   of the manifold $G$, we denote by $\Phi^!TG$, the pullback bundle of the tangent bundle $TG$ along the diffeomorphism $\Phi : G\rightarrow G$.
Triple $(\Phi^!TG,\Phi,Ad_{\Phi^*})$ is an example of hom-bundles, where
 $Ad_{\Phi^*}:\Gamma(\Phi^!TG)\rightarrow \Gamma(\Phi^!TG)$ is given by
\[
Ad_{\Phi^*}(x)=\Phi^*\circ x\circ(\Phi^*)^{-1},
\]
for any $x\in\Gamma(\Phi^!TG)$ \cite{CLS}. Note that  $\Gamma(\Phi^!TG)$ can be identified with  $ Der_{\Phi^*,\Phi^*}({C^\infty(G))}$, i.e.
\[
x(fg)=x(f)\Phi^*(g)+\Phi^*(f)x(g),\ \ \ \ \ \ \ \ \ \forall x\in\Gamma(\Phi^!TG), \forall f,g\in C^\infty(G).
\]
%Consider the tangent bundle $TG$   of the manifold $G$, we denote by $\Phi^!TG$, the pullback bundle of the tangent bundle $TG$ along the diffeomorphism $\Phi : G\rightarrow G$.  There is a one-to-one correspondence between the space of sections of the tangent bundle $TG$ (i.e., $\Gamma(TG)$) and the space of sections of the pullback bundle $\Phi^!TG$ (i.e., $\Gamma(\Phi^!TG)$)\cite{JMS}.
 The linear map $Ad_{\Phi^*}:\Gamma(\Phi^!TG)\rightarrow  \Gamma(\Phi^!TG)$ given in a  hom-bundle $(\Phi^!TG\rightarrow G, \Phi,Ad_{\Phi^*})$  can be extended to a linear
map from $\Gamma(\wedge^k\Phi^!TG)$ to $\Gamma(\wedge^k\Phi^!TG)$ for which we use the same notation $Ad_{\Phi^*}$ via
\[
Ad_{\Phi^*}(x)=Ad_{\Phi^*}(x_1)\wedge\ldots\wedge Ad_{\Phi^*}(x_k),\ \ \ \ \ \ \ \forall x=x_1\wedge\ldots\wedge x_k\in\Gamma(\wedge^k\Phi^!TG).
\]
If we denote the inverses of $\Phi$ and $Ad_{\Phi^*}$  by $\Phi^{-1}$ and $Ad_{(\Phi^*)^{-1}}$, respectively, it is easy to see that
\[
Ad_{(\Phi^*)^{-1}}(fX)=(\Phi^*)^{-1}(f)Ad_{(\Phi^*)^{-1}}(X),\ \ \ \ \forall f\in C^\infty(G), x\in\Gamma(\wedge^k\Phi^!TG).
\]
So $(\Phi^!TG, \Phi^{-1},Ad_{(\Phi^*)^{-1}})$ is a hom-bundle.
In the sequel, we denote by $\Phi^!T^*G$
the dual of the pullback bundle $\Phi^!TG$.
 We consider $Ad_{\Phi^*}^\dagger:\Gamma(\wedge^k\Phi^!T^*G)\rightarrow\Gamma(\wedge^k\Phi^!T^*G)  $ that is defined by 
\[
(Ad_{\Phi^*}^\dagger(\xi))(x)=\Phi^*\xi (Ad_{(\Phi^*)^{-1}}(x)),\ \ \ \ \ \forall x\in\Gamma(\wedge^k\Phi^!TG), \xi\in \Gamma(\wedge^k\Phi^!T^*G).
\]
From the above equation, we can conclude the following 
\[
Ad_{\Phi^*}^\dagger(f\xi)=\Phi^*(f)Ad_{\Phi^*}^\dagger(\xi).
\]
Therefore $(\wedge^k\Phi^!T^*G,\Phi,Ad_{\Phi^*}^\dagger)$ is a hom-bundle.
\begin{theorem}
	Let $G$ be a smooth manifold. Then $(\Gamma(\Phi^!TG),[\cdot,\cdot]_\Phi,Ad_{\Phi^*})$ is a Hom-Lie algebra, where the Hom-Lie bracket $[\cdot,\cdot]_\Phi$ and the map $\phi : \Gamma(\Phi^!TG) \rightarrow \Gamma(\Phi^!TG)$ are defined as follows:
	\begin{align*}
	\phi(x)&=Ad_{\Phi^*}(x)=\Phi^*\circ x \circ (\Phi^{-1})^* , \\
	[x,y]_\Phi&=\Phi^* \circ x \circ (\Phi^{-1})^* \circ y \circ (\Phi^{-1})^*-\Phi^*  \circ y \circ (\Phi^{-1})^* \circ x \circ (\Phi^{-1})^* ,
	\end{align*}
	for any $x,y\in \Gamma(\Phi^!TG)$.
\end{theorem}
The above Theorem is Theorem 3.6 of \cite{JMS} only the difference is that
in \cite{JMS}, authors consider the following relations for $Ad_{\Phi^*}$ and $[\cdot,\cdot]_\Phi$:
\begin{align*}
	\phi(x) =& (\Phi^{-1})^*\circ x \circ  \Phi^*=Ad_{(\Phi^{-1})^*}, \\
	[x,y]_\Phi =& (\Phi^{-1})^* \circ x \circ (\Phi^{-1})^* \circ y \circ \Phi^*- (\Phi^{-1})^*\circ y \circ (\Phi^{-1})^* \circ x \circ \Phi^*.
	\end{align*}
\begin{definition}\cite{JMS}\label{JMS}
	 Let $(G,\diamond,e_\Phi,\Phi)$ be a Hom-Lie group. A smooth section $x\in \Gamma(\Phi^!TG)$ is called left-invariant if $x$ satisfies the following equation:
	 \begin{align*}
	 x_a=(l_a\circ\Phi^{-1})_{*_{e_\Phi}}(x_{e_\Phi}),
	 \end{align*}
	 or
	 \[
	 x(f)(a) = x(f \circ l_a \circ \Phi^{-1})(e_\Phi),
	 \]
	 for any $a\in G$ and $f\in C^\infty (G)$, where $l_a:G\rightarrow G$ is a smooth map given by $l_ab=a\diamond b$ for any $b\in G$.
	 
\end{definition}
\begin{theorem}\cite{JMS}
The space $\Gamma_L(\Phi^!TG)$ of left-invariant sections of the pullback bundle $\Phi^!TG$ is a Hom-Lie subalgebra of the Hom-Lie algebra $(\Gamma(\Phi^!TG), [\cdot, \cdot]_\Phi, \phi)$.
\end{theorem}
\begin{remark}\label{Rem}
Let $(G, \diamond, e_{\Phi}, \Phi)$ be a Hom-Lie group and $\mathfrak{g}^!$ be the fibre of $e_{\Phi}$ in the pullback bundle $\Phi^!TG$. Then $\Phi^!T_{e_\Phi}G=\mathfrak{g}^!$ and also $\mathfrak{g}^!$ is in one-to-one correspondence with $\Gamma_L(\Phi^!TG)$ (see Lemma 3.11 in \cite{JMS}, for more details).
Moreover, defining a bracket $[\cdot,\cdot]_{\mathfrak{g^!}}$ and a vector space isomorphism $\phi_{\mathfrak{g^!}}:\mathfrak{g^!}\rightarrow \mathfrak{g^!}$ by 
\begin{align*}
[x(e_\Phi),y(e_\Phi) ]_{\mathfrak{g^!}}=&[x,y]_\Phi(e_\Phi),\\
\phi_{\mathfrak{g^!}}(x(e_\Phi))=&(\phi(x))(e_\Phi),
\end{align*}
for all $x,y\in \Gamma_L(\Phi^!TG)$. The triple $(\mathfrak{g^!}, [\cdot,\cdot]_{\mathfrak{g^!}},\phi_{\mathfrak{g^!}} )$ is a Hom-Lie algebra that is isomorphic to the Hom-Lie algebra $(\Gamma_L(\Phi^!TG), [\cdot,\cdot]_\Phi,\phi)$.
\end{remark}
\begin{definition}\label{JMS1}
	Let $(G,\diamond,e_\Phi,\Phi)$ be a Hom-Lie group. A smooth section $\omega\in \Gamma(\Phi^!T^*G)$ is called left-invariant if
	\begin{align*}
	\omega_a=(l_a\circ\Phi^{-1})^*{_{e_\Phi}}(\omega_{e_\Phi}),
	\end{align*}
	for any $a\in G$, where $((l_a\circ\Phi^{-1})^*\omega)(x)=\omega((l_a\circ\Phi^{-1})_*x)$. We denote by $\Gamma_L(\Phi^!T^*G)$, the space of these left-invariant sections.
	
\end{definition}
A {\it tensor $T$ of type $(p,q)$ or a $(p, q)$-tensor}  on a Hom-Lie group $(G, \diamond, e_{\Phi}, \Phi)$ is a $\Phi^*$-function linear mapping
\[
T:\underset{p}{\underbrace{\Gamma(\Phi^!T^*G)\times \cdots \times \Gamma(\Phi^!T^*G)}}\times \underset{q}{\underbrace{\Gamma(\Phi^!TG)\times \cdots \times \Gamma(\Phi^!TG)}}\rightarrow C^\infty (G).
\]
This means that given $y_1,\ldots, y_q\in\Gamma(\Phi^!TG)$ and $\theta^1, \cdots, \theta^p\in \Gamma(\Phi^!T^*G)$, $T(\theta^1, \cdots, \theta^p, y_1, \cdots, y_q)$ is a differentiable function on $G$ and that $T$ is $\Phi^*$-function linear  in each argument, i.e., 
\begin{align*}
T(\theta^1, \cdots, \theta^p, y_1, \cdots,fx+y,\cdots, y_q)=&\Phi^*(f)T(\theta^1, \cdots, \theta^p, y_1, \cdots,x,\cdots, y_q)\\
&+T(\theta^1, \cdots, \theta^p, y_1, \cdots,y,\cdots, y_q),
\end{align*} 
and 
\begin{align*}
T(\theta^1, \cdots,f\theta+\eta \cdots,\theta^p, y_1, \cdots, y_q)=&\Phi^*(f)T(\theta^1, \cdots, \theta \cdots,\theta^p, y_1, \cdots, y_q)\\
&+T(\theta^1, \cdots,\eta \cdots, \theta^p, y_1, \cdots, y_q),
\end{align*}
for any $x,y\in\Gamma(\Phi^!TG)$ and $\theta,\eta\in\Gamma(\Phi^!T^*G)$.
We denote the set of tensors of
type $(p,q)$
by $\mathcal{T}_{q}^{p}(\Gamma(\Phi^!TG))$. Also,  the set of tensors of
type $(p,q)$ on $(\mathfrak{g^!}, [\cdot,\cdot]_{\mathfrak{g^!}},\phi_{\mathfrak{g^!}} )$
is denoted 
by $\mathcal{T}_{q}^{p}(\mathfrak{g^!})$.
\begin{lemma}
	Let  $(G, \diamond, e_{\Phi}, \Phi)$ be a Hom-Lie group. There is an   isomorphism between $\mathcal{T}_{q}^{p+1}(\Phi^!TG)$ and the space of $\Phi^*$-function linear maps 
	\[
	\underset{p}{\underbrace{\Gamma(\Phi^!T^*G)\times \cdots \times \Gamma(\Phi^!T^*G)}}\times \underset{q}{\underbrace{\Gamma(\Phi^!TG)\times \cdots \times \Gamma(\Phi^!TG)}}\rightarrow \Gamma(\Phi^!TG),
	\]
	which we denote this space by $L_{p,q}(\Phi^!T^*G,\Phi^!TG;\Phi^!TG)$.
\end{lemma}
\begin{proof}
	We consider the map $\Psi:L_{p,q}(\Phi^!T^*G,\Phi^!TG;\Phi^!TG)\rightarrow \mathcal{T}_{q}^{p+1}(\Phi^!TG)$ given by
	\begin{align*}
	\Psi(F)(\omega,\theta^1, \cdots, \theta^p, x_1, \cdots, x_q)=&\Phi^*
	(\omega(F(Ad_{(\Phi^*)^{-2}}^\dagger(\theta^1), \cdots, Ad_{(\Phi^*)^{-2}}^\dagger(\theta^p),Ad_{(\Phi^*)^{-2}}( x_1), \cdots,\\
	& Ad_{(\Phi^*)^{-2}}( x_q)))).
	\end{align*}
	It is easy to see that $\Psi$ is an  isomorphism. 
\end{proof}
\begin{definition}
	A tensor $T\in\mathcal{T}_{q}^{p}(\Gamma(\Phi^!TG))$ is called  left-invariant if we have 
	\begin{align*}
&	T((l_a\circ\Phi^{-1})^*{_{e_\Phi}}((\theta^1)_{e_\Phi}),\cdots, (l_a\circ\Phi^{-1})^*{_{e_\Phi}}((\theta^p)_{e_\Phi}), (l_a\circ\Phi^{-1})_{*_{e_\Phi}}((x_1)_{e_\Phi}),\cdots,(l_a\circ\Phi^{-1})_{*_{e_\Phi}}((x_q)_{e_\Phi}))\\
&\ \ \ \ \ =T((\theta^1)_{e_\Phi},\cdots,(\theta^p)_{e_\Phi}
	,(x_1)_{e_\Phi},\cdots,(x_q)_{e_\Phi}),
	\end{align*}
	for any $\theta^1,\cdots,\theta^p\in\Gamma(\Phi^!T^*G)$ and $x_1,\cdots,x_q\in \Gamma(\Phi^!TG)$.
\end{definition}
\begin{remark}\label{IMP}
If $T$ is a left-invariant tensor of type $(p, q)$ on a Hom-Lie group $(G, \diamond, e_{\Phi}, \Phi)$, then using Definitions \ref{JMS} and \ref{JMS1} we get 
\begin{align*}
T((\theta^1)_a,\cdots,(\theta^p)_a
	,(x_1)_a,\cdots,(x_q)_a),=T((\theta^1)_{e_\Phi},\cdots,(\theta^p)_{e_\Phi}
	,(x_1)_{e_\Phi},\cdots,(x_q)_{e_\Phi}),\ \ \ \forall a\in G,
\end{align*}
where $\theta^1,\cdots,\theta^p\in\Gamma(\Phi^!T^*G)$ and $x_1,\cdots,x_q\in \Gamma(\Phi^!TG)$. So, the restriction of $T$ on left-invariant sections is constant.
\end{remark}
%%%%%%%%%%%%%%%%%%%%%%%%%%%%%%%%%%%%%%%%%%%%%%%%%%%%
%In this section, we recall some basic definitions on hom-category.  
%%%%%%%%%%%%%%%%%%%%%%%%%%%%%%%%%%%%%%%%%%%%%%%
\section{ Norden structures on Hom-Lie algebras}
\begin{definition}
An almost complex structure on a Hom-Lie group $(G, \diamond, e_{\Phi}, \Phi)$ is a $(1, 1)$-tensor field $J:\Phi^!TG\rightarrow\Phi^!TG$ such that $(Ad_{\Phi^*}\circ J)^2=-Id_{_{\Phi^!TG}}$ and $Ad_{\Phi^*}\circ J=J\circ Ad_{\Phi^*}$. Moreover, if for all $x\in\Gamma_L(\Phi^!TG)$ we have $J\circ x\in\Gamma_L(\Phi^!TG)$, then $J$ is called a left-invariant almost complex structure. The  left-invariant almost complex structure is integrable (left-invariant complex structure) if the Nijenhuis torsion $N_{Ad_{\Phi^*}\circ J}$ of $Ad_{\Phi^*}\circ J$ defined by
\begin{align}\label{M1}
	N_{{ \phi_\mathfrak{g^!}}\circ J}(x,y)&=[(Ad_{\Phi^*}\circ J)x, (Ad_{\Phi^*}\circ J)y]-Ad_{\Phi^*}\circ J[({{ \phi_\mathfrak{g^!}}\circ J})x, y]\nonumber\\
	&\ \ \ -Ad_{\Phi^*}\circ J[x,(Ad_{\Phi^*}\circ J)y]-[x,y],
\end{align}
for any $x, y\in \Gamma_L(\Phi^!TG)$, is zero. 
\end{definition}
According to Remark \ref{Rem}, we can define a left-invariant almost complex structure on a Hom-Lie group as follows:
\begin{definition}
	A left invariant almost complex structure on the Hom-Lie group $(G, \diamond, e_{\Phi}, \Phi)$ (or an almost complex structure on a Hom-Lie algebra $(\mathfrak{g^!}, [\cdot  , \cdot ]_\mathfrak{g^!}, { \phi_\mathfrak{g^!}})$ ), is an isomorphism  $J :\mathfrak{g^!}\rightarrow\mathfrak{g^!}$ that  satisfies  
	$
	(\phi_{\mathfrak{g}^!}\circ J)^2 = -Id_{\mathfrak{g}^!}$ and ${ \phi_{\mathfrak{g}^!}}\circ J=J\circ{ \phi_{\mathfrak{g^!}}}$. 
\end{definition}
	We denote an almost complex Hom-Lie algebra by  $(\mathfrak{g}^!, [\cdot  , \cdot ], { \phi_\mathfrak{g^!}}, J)$.
If the $n$-dimensional Hom-Lie algebra $\mathfrak{g^!}$ admits an almost complex structure $J$, then
\[
(det(\phi_{\mathfrak{g^!}}\circ J))^2=det((\phi_{\mathfrak{g^!}}\circ J)^2)=det(-Id_{\mathfrak{g^!}})=(-1)^n,
\]
which implies that $n$ is even.
\begin{definition}\label{LM1}
Let  $(G, \diamond, e_{\Phi}, \Phi)$  be a Hom-Lie group. A $(0, 2)$-tensor symmetric and non-degenerate $\langle\cdot ,\cdot \rangle:\Gamma(\Phi^! TG)\times\Gamma(\Phi^! TG)\rightarrow C^\infty(G)$  is called a pseudo-Riemannian metric if it is a left-invariant tensor, i.e., 
\[
\langle Ad_{\Phi^*}(x), Ad_{\Phi^*}(y)\rangle =
\Phi^*\langle x,y\rangle, 
\]
for any $x,y\in \Gamma(\Phi^! TG)$. Moreover, $\langle\cdot ,\cdot\rangle$
 is called a left-invariant  pseudo-Riemannian metric if 
\begin{align*}
	\langle(l_a\circ\Phi^{-1})_{*_{e_\Phi}}(x_{e_\Phi}), (l_a\circ\Phi^{-1})_{*_{e_\Phi}}(y_{e_\Phi})\rangle_a=\langle x_{e_\Phi}, y_{e_\Phi}\rangle_{e_\Phi},
\end{align*} 
for all $x, y\in\Gamma_L(\Phi^!TG)$.
\end{definition}
\begin{remark}
According to Remark \ref{Rem} and Definition \ref{LM1}, we conclude that a left-invariant pseudo-Riemannian metric on a Hom-Lie group $(G, \diamond, e_{\Phi}, \Phi)$  (or a  pseudo-Riemannian metric on a Hom-Lie algebra $(\mathfrak{g}^!, [\cdot  , \cdot ], { \phi_\mathfrak{g^!}})$) is a $\Phi^*$-function linear symmetric non-degenerate form $\langle\cdot,\cdot\rangle$ such that 
\begin{equation*}
	\langle\phi_\mathfrak{g^!}(x), \phi_\mathfrak{g^!}(y)\rangle=\langle x, y\rangle,
\end{equation*}
fro any $x,y\in \mathfrak{g^!}$. In this case,   $(G, \diamond, e_{\Phi}, \Phi, \langle\cdot,\cdot\rangle)$ (or $(\mathfrak{g}^!, [\cdot  , \cdot ], { \phi_\mathfrak{g^!}}), \langle\cdot,\cdot\rangle$) is called a left-invariant pseudo-Riemannian Hom-Lie group
(or a pseudo-Riemannian Hom-Lie algebra).
Also, if $\phi_\mathfrak{g^!}^2=Id_\mathfrak{g^!}$, then we have 
	\begin{equation}\label{EL1}
	\langle\phi_\mathfrak{g^!}(x), y\rangle=\langle x, \phi_\mathfrak{g^!}(y)\rangle,
	\end{equation}
	for any $x,y\in \mathfrak{g^!}$. Changing $\mathfrak{g^!}$ to $\mathfrak{g}$ and $\phi_{\mathfrak{g^!}}$ to $\phi_\mathfrak{g}$, we have the definition of a pseudo-Riemannian metric on a Hom-Lie algebra $(\mathfrak{g}, [\cdot,\cdot]_\mathfrak{g}, \phi_\mathfrak{g})$ \cite{PN}.
\end{remark}
\begin{definition}\label{+}
A linear connection on a Hom-Lie group $(G, \diamond, e_{\Phi}, \Phi)$  is a map
\begin{align*}
\begin{cases}
&\nabla:\Gamma(\Phi^!TG)\times\Gamma(\Phi^!TG))\rightarrow \Gamma(\Phi^!TG))\\
&\hspace{1cm}(x, y)\longrightarrow\nabla_xy
\end{cases},
\end{align*}
with the following properties:\\
i) $\nabla_{x+y}z=\nabla_xz+\nabla_yz$,\\
ii) $\nabla_{fx}y=\Phi^*(f)\nabla_xy$,\\
iii) $\nabla_x(y+z)=\nabla_xy+\nabla_xz$,\\
iv) $\nabla_x(fy)=\Phi^*(f)\nabla_xy+Ad_{\Phi^*}(x)(f)Ad_{\Phi^*}y$,\\
for any $f\in C^{\infty}(G ), x,y,z\in\Gamma(\Phi^!TG) $.
\end{definition}
\begin{theorem}\label{**}
Let $(G, \diamond, e_{\Phi}, \Phi, \langle\cdot,\cdot\rangle)$ be a pseudo-Riemannian Hom-Lie group. Then there exists the unique connection $\nabla$ on it which is characterized by the following properties:

(i)\ \ $\nabla$ is symmetric, i.e., 
\begin{align*}
[x, y]_\Phi=\nabla_xy-\nabla_yx,
\end{align*}

(ii)\ \ $\nabla$ is compatible with  $\langle\cdot,\cdot\rangle$ , i.e., 
\[
Ad_{\Phi^*}(x)<y,z>=<\nabla_xy,Ad_{\Phi^*}(z)>+<Ad_{\Phi^*}(y),\nabla_xz>,
\]
for any $x,y ,z \in \Gamma (\Phi^!TG)$. This connection is called the Hom-Levi-Civita connection.
\end{theorem}
\begin{proof}
Considering $\nabla$ as 
\begin{align}\label{*}
2\langle\nabla_xy,Ad_{\Phi^*}(z)\rangle=&Ad_{\Phi^*}(x)\langle y,z\rangle
+Ad_{\Phi^*}(y)\langle z,x\rangle-Ad_{\Phi^*}(z)\langle x,y\rangle\\
&+\langle [x,y]_\Phi,Ad_{\Phi^*}(z)\rangle+\langle [z,x]_\Phi,Ad_{\Phi^*}(y)\rangle+\langle [z,y]_\Phi,Ad_{\Phi^*}(x)\rangle\nonumber,
\end{align}
it is easy to see that $\nabla$ satisfies in the properties of Definition. Also, the above equation gives us 
\[
\langle\nabla_xy-\nabla_yx,Ad_{\Phi^*}(z) \rangle=
\langle[x,y]_\Phi,Ad_{\Phi^*}(z) \rangle,
\]
which implies (i). In the similar way, we can deduce (ii). Finally (\ref{*})  shows the uniqueness of $\nabla$.
\end{proof}
Formula (\ref{*}) given in the above theorem is called  Koszul's formula.
\begin{definition}\label{++}
A linear connection $\nabla$ on a Hom-Lie group $(G, \diamond, e_{\Phi}, \Phi)$ is called left-invariant if for all $x,y\in \Gamma_L(\Phi^!TG)$ we have $\nabla_xy\in \Gamma_L(\Phi^!TG)$.
\end{definition}
\begin{remark}\label{&&}
If $\langle\cdot,\cdot\rangle$  is a left-invariant pseudo-Riemannian metric on a Hom-Lie group $(G, \diamond, e_{\Phi}, \Phi)$, then using Remark \ref{IMP},
$\langle x,y\rangle$  is a constant for all $x, y\in \Gamma_L(\Phi^!TG)$, and so we get $Ad_{\Phi^*}(z)\langle x,y\rangle=0$. Therefore   Koszul's formula reduces to the following 
\begin{align*}
2\langle\nabla_xy,Ad_{\Phi^*}(z)\rangle=\langle [x,y]_\Phi,Ad_{\Phi^*}(z)\rangle+\langle [z,x]_\Phi,Ad_{\Phi^*}(y)\rangle+\langle [z,y]_\Phi,Ad_{\Phi^*}(x)\rangle,
\end{align*}
for any $x,y,z\in\Gamma_L(\Phi^!TG)$. 
\end{remark}
\begin{remark}
According to Definitions \ref{+} and \ref{++}, we can rewrite the definition of left-invariant connection on a Hom-Lie group (or a connection on a Hom-Lie algebra $(\mathfrak{g^!}, [\cdot,\cdot]_{\mathfrak{g^!}}, \phi_{\mathfrak{g^!}})$) as follows:\\
A left-invariant connection on a Hom-Lie group $(G, \diamond, e_{\Phi}, \Phi)$ (or a connection on a Hom-Lie algebra $(\mathfrak{g^!}, [\cdot,\cdot]_{\mathfrak{g^!}}, \phi_{\mathfrak{g^!}})$) is a linear map
\begin{align*}
\begin{cases}
&\nabla:\mathfrak{g^!}\times\mathfrak{g^!}\rightarrow \mathfrak{g^!}\\
&\hspace{.2cm}(x, y)\longrightarrow\nabla_xy
\end{cases},
\end{align*}
with the following properties:
\[
 \nabla_{fx}y=\Phi^*(f)\nabla_xy,\ \ \ \ 
 \nabla_x(fy)=\Phi^*(f)\nabla_xy+\phi_{\mathfrak{g^!}}(x)(f)\phi_{\mathfrak{g^!}}(y),
 \]
for any $f\in C^{\infty}(G)$ and $ x,y,z\in\mathfrak{g^!} $. Also, from Theorem \ref{**} and Remark \ref{&&} we deduce that the linear connection $\nabla$ given by
\begin{equation}\label{Koszul}
	2\langle \nabla_x y,{ \phi_\mathfrak{g^!}}(z)\rangle=\langle[x,y]_\mathfrak{g^!},{ \phi_\mathfrak{g^!}}(z)\rangle+\langle[z,y]_\mathfrak{g^!},{ \phi_\mathfrak{g^!}}(x)\rangle+\langle[z,x]_\mathfrak{g^!},{ \phi_\mathfrak{g^!}}(y)\rangle,\ \ \ \forall x,y,z\in\mathfrak{g^!}, 
\end{equation}
is the unique left-invariant connection (or the unique connection) on a left-invariant pseudo-Riemannian Hom-Lie group  $(G, \diamond, e_{\Phi}, \Phi, \langle\cdot,\cdot\rangle)$ (or pseudo-Riemannian Hom-Lie algebra $(\mathfrak{g^!}, [\cdot,\cdot]_{\mathfrak{g^!}}, \phi_{\mathfrak{g^!}})$)
such that 
\begin{align}\label{twoo}
(i)\ \ [x,y]_{\mathfrak{g^!}}=\nabla_x y-\nabla_y x,\ \ \ \
(ii)\ \ \langle \nabla_x y,{ \phi_\mathfrak{g^!}}(z)\rangle=-\langle{ \phi_\mathfrak{g^!}}(y),\nabla_x z\rangle,
\end{align}
for any $x,y,z\in \mathfrak{g^!}$. This connection is called the Hom-Levi-Civita connection
on  $(G, \diamond, e_{\Phi}, \Phi, \langle\cdot,\cdot\rangle)$ (or $(\mathfrak{g}^!, [\cdot  , \cdot ]_{\mathfrak{g^!}}, { \phi_\mathfrak{g^!}}), \langle\cdot,\cdot\rangle$). Changing 
$\mathfrak{g^!}$ to $\mathfrak{g}$ and $\phi_{\mathfrak{g^!}}$ to $\phi_\mathfrak{g}$, we have the definition of a linear connection (also, Hom--Levi-Civita connection) for any arbitrary Hom-Lie algebra $(\mathfrak{g}, [\cdot,\cdot]_\mathfrak{g}, \phi_\mathfrak{g}, \langle\cdot,\cdot\rangle))$.
\end{remark}
 The total covariant derivative of a $(p,q)$-tensor $T$  on a Hom-Lie group $(G, \diamond, e_{\Phi}, \Phi, \langle\cdot,\cdot\rangle)$ is a tensor of   order  $(p,q+1)$  given by 
\begin{align}\label{IMP1}
(\nabla_xT)(\theta^1, \cdots, \theta^p, y_1, \cdots, y_q)&=\nabla_xT(\theta^1, \cdots, \theta^p, y_1, \cdots, y_q)\\
&-\sum_{i=1}^{p}T(Ad_{\Phi^*}^\dagger(\theta^1), \cdots, \nabla_x\theta^i, \cdots, Ad_{\Phi^*}^\dagger(\theta^p), Ad_{\Phi^*}(y_1), \cdots, Ad_{\Phi^*}(y_q))\nonumber\\
&-\sum_{i=1}^{q}T(Ad_{\Phi^*}^\dagger(\theta^1), \cdots, Ad_{\Phi^*}^\dagger(\theta^p),Ad_{\Phi^*}( y_1),  \cdots, \nabla_xy_i,\cdots,Ad_{\Phi^*}(y_q))\nonumber,
\end{align}
for all $\theta^1,\cdots,\theta^p\in\Gamma(\Phi^!T^*G)$ and $x_1,\cdots,x_q\in \Gamma(\Phi^!TG)$. If $\nabla$ is a left-invariant connection on a Hom-Lie group $(G, \diamond, e_{\Phi}, \Phi, \langle\cdot,\cdot\rangle)$, then using Remark \ref{IMP}, $T(\theta^1, \cdots, \theta^p, y_1, \cdots, y_q)$ is constant, for all $\theta^1,\cdots,\theta^p\in\Gamma(\Phi^!T^*G)$ and $x_1,\cdots,x_q\in \Gamma(\Phi^!TG)$. So, 
\begin{align*}
\nabla_xT(\theta^1, \cdots, \theta^p, y_1, \cdots, y_q)=x(T(\theta^1, \cdots, \theta^p, y_1, \cdots, y_q))=0.
\end{align*}
Therefore, (\ref{IMP1}) reduces to the following:
\begin{align}\label{IMP2}
(\nabla_xT)(\theta^1, \cdots, \theta^p, y_1, \cdots, y_q)&=-\sum_{i=1}^{p}T(Ad_{\Phi^*}^\dagger(\theta^1), \cdots, \nabla_x\theta^i, \cdots, Ad_{\Phi^*}^\dagger(\theta^p), Ad_{\Phi^*}(y_1), \cdots, Ad_{\Phi^*}(y_q))\\
&\ \ \ -\sum_{i=1}^{q}T(Ad_{\Phi^*}^\dagger(\theta^1), \cdots, Ad_{\Phi^*}^\dagger(\theta^p),Ad_{\Phi^*}( y_1),  \cdots, \nabla_xy_i,\cdots,Ad_{\Phi^*}(y_q))\nonumber,
\end{align}
\begin{remark}
Given  $(\mathfrak{g^!})^*$ as the dual 
space of $\mathfrak{g^!}$,
we denote by  $\phi_{{\mathfrak{g^!}}}^*$ the dual map of the endomorphism $\phi_{{\mathfrak{g^!}}}$.
In this case, according to (\ref{IMP2}), the total covariant derivative of $T\in \mathcal{T}_{q}^{p}(\mathfrak{g^!})$  is a tensor of   order  $(p,q+1)$  as follows 
 \begin{align}\label{L1}
(\nabla_xT)(\theta^1, \cdots, \theta^p,y_1, \cdots, y_q)&=-\sum_{i=1}^{p}T(\phi_{{\mathfrak{g^!}}}^*(\theta^1), \cdots, \nabla_x\theta^i, \cdots, \phi_{{\mathfrak{g^!}}}^*(\theta^p), \phi_{{\mathfrak{g^!}}}(y_1), \cdots, \phi_{{\mathfrak{g^!}}}(y_q))\\
&\ \ \ -\sum_{i=1}^{q}T(\phi_{{\mathfrak{g^!}}}^*(\theta^1), \cdots, \phi_{{\mathfrak{g^!}}}^*(\theta^p),\phi_{{\mathfrak{g^!}}}( y_1),  \cdots, \nabla_x y_i,\cdots,\phi_{{\mathfrak{g^!}}}( y_q))\nonumber,
\end{align}
for all $\theta^1,\cdots,\theta^p\in(\mathfrak{g^!})^*$ and $x_1,\cdots,x_q\in \mathfrak{g^!}$.
\end{remark}
\begin{definition}\label{L100}
Let  $(G, \diamond, e_{\Phi}, \Phi, J)$  be a left-invariant almost complex Hom-Lie group(or $(\mathfrak{g^!}, [\cdot,\cdot]_\mathfrak{g^!},\phi_{\mathfrak{g^!}} ,J)$ be an almost complex Hom-Lie algebra).
A left-invariant  pseudo-Riemannian metric $\langle\cdot,\cdot\rangle$ on 
$G$(or a  pseudo-Riemannian metric $\langle\cdot,\cdot\rangle$ on $\mathfrak{g^!}$) is called Norden metric (anti-Hermitian, B-metric) if 
		\begin{equation*}
	\langle({ \phi_\mathfrak{g^!}}\circ J)x, ({ \phi_\mathfrak{g^!}}\circ J)y\rangle=-\langle x , y\rangle, 
	\end{equation*}
	or equivalenty
	\begin{equation}\label{L14}
	\langle({ \phi_\mathfrak{g^!}}\circ J)x, y\rangle=\langle x , ({ \phi_\mathfrak{g^!}}\circ J)y\rangle, 
	\end{equation}
for all $x,y \in \mathfrak{g^!}$. In this case, we say that $(G, \diamond, e_{\Phi}, \Phi, J, \langle \cdot,\cdot\rangle)$ is a left-invariant almost Norden Hom-Lie group (or, $(\mathfrak{g^!}, [\cdot,\cdot]_\mathfrak{g^!},\phi_\mathfrak{g^!}, J, \langle\cdot,\cdot\rangle)$ is an almost Norden Hom-Lie algebra).
 Moreover, if ${{ \phi_\mathfrak{g^!}}\circ J}$ is integrable, then the pair $(J, \langle\cdot,\cdot\rangle)$ is called left-invariant Norden structure on $G$ (or, Norden structure on  $\mathfrak{g^!}$).
\end{definition}
%%%%%%%%%%%%%%%%%%%%%%%%%%%%%%%%%%%%%%%%%%%%%%%%%
\begin{remark}
All concepts defined on the Hom-Lie algebra $(\mathfrak{g^!}, [\cdot,\cdot]_\mathfrak{g^!},\phi_\mathfrak{g^!})$ of a Hom-Lie group 
$(G, \diamond, e_{\Phi}, \Phi)$   hold for an arbitrary Hom-Lie algebra. So, we study these concepts on Hom-Lie algebras as examples.
\end{remark}
%According to the above definition,
% we present  examples of  $4$-dimensional spaces.
 \begin{example}\label{L10}
 	We consider a $4$-dimensional linear space $\mathfrak{g}$ with an arbitrary basis $\{e_1,e_2,e_3,e_4\}$. We define the bracket
 	and linear map $\phi_\mathfrak{g}$ on $\mathfrak{g}$ as follows
 	\begin{align*}
 	[e_1,e_4]_\mathfrak{g}=&ae_2,\ \ \ [e_2,e_3]_\mathfrak{g}=ae_1,
 	\end{align*}
 	and
 	\[
 	\phi_\mathfrak{g}(e_1)=-e_2,\ \ \ \phi_\mathfrak{g}(e_2)=-e_1,\ \ \ \phi_\mathfrak{g}(e_3)=e_4, \ \ \phi_\mathfrak{g}(e_4)=e_3.
 	\]
 	The above bracket is not a Lie bracket on $\mathfrak{g}$ if $a\neq 0$,  because
 	\[
 	[e_2,[e_3,e_4]_\mathfrak{g}]_\mathfrak{g}+[e_3,[e_4,e_2]_\mathfrak{g}]_\mathfrak{g}+[e_4,[e_2,e_3]_\mathfrak{g}]_\mathfrak{g}=[e_4,ae_1]_\mathfrak{g}+[e_3,ae_1]_\mathfrak{g}=-ae_2.
 	\]
 	It is easy to see that
 	\begin{align*}
 	[\phi_\mathfrak{g}(e_1),\phi_\mathfrak{g}(e_4)]_\mathfrak{g}=-ae_1=\phi_\mathfrak{g}([e_1, e_4]_\mathfrak{g}),\ \ [\phi_\mathfrak{g}(e_2),\phi_\mathfrak{g}(e_3)]_\mathfrak{g}=-ae_2=\phi_\mathfrak{g}([e_2, e_3]_\mathfrak{g}),
 	\end{align*}
 	i.e., $\phi_\mathfrak{g}$ is the algebra morphism. Also, we can deduce
 	\begin{equation*}
 	[\phi_\mathfrak{g}(e_i),[e_j,e_k]_\mathfrak{g}]_\mathfrak{g}+[\phi_\mathfrak{g}(e_j),[e_k,e_i]_\mathfrak{g}]_\mathfrak{g}+[\phi_\mathfrak{g}(e_k),[e_i,e_j]_\mathfrak{g}]_\mathfrak{g}=0,\ \ \ i,j,k=1, 2, 3, 4.
 	\end{equation*}
 	Thus $(\mathfrak{g}, [\cdot,\cdot]_\mathfrak{g}, \phi_\mathfrak{g})$ is a Hom-Lie algebra.
 	We define the metric $\langle\cdot,\cdot\rangle$ of $\mathfrak{g}$ as follows
 	\[
 	\begin{bmatrix}
 	A&B&C&D\\
 	B&A&-D&-C\\
 	C&-D&-A&B\\
 	D&-C&B&-A
 	\end{bmatrix}
 	.
 	\]
 	It follows that  $\langle\phi_\mathfrak{g}(e_i),\phi_\mathfrak{g}(e_j)\rangle=\langle e_i,e_j\rangle$, for all $i,j=1,2,3,4$, hence
  (i) of (\ref{twoo})  holds and so $(\mathfrak{g}, [\cdot, \cdot]_\mathfrak{g}, \phi_{\mathfrak{g}}, \langle\cdot,\cdot\rangle)$ is a pseudo-Riemannian Hom-Lie algebra. Assume that isomorphism $J$ is determined as
 	\[
 	J(e_1)=-e_3,\ \ \ J(e_2)=e_4,\ \ \ J(e_3)=e_1, \ \ J(e_4)=- e_2,
 	\]
 	then we have
 	\begin{align*}
 	(J\circ\phi_\mathfrak{g})e_1=&-e_4=(\phi_\mathfrak{g}\circ J)e_1,\ \ \ (J\circ\phi_\mathfrak{g})e_2=e_3=(\phi_\mathfrak{g}\circ J)e_2,\\
 	(J\circ\phi_\mathfrak{g})e_3=&-e_2=(\phi_\mathfrak{g}\circ J)e_3,\ \ \ (J\circ\phi_\mathfrak{g})e_4=e_1=(\phi_\mathfrak{g}\circ J)e_4.\\
 	\end{align*}
 	Thus $J$ is an almost complex structure on $(\mathfrak{g}, [\cdot, \cdot]_\mathfrak{g}, \phi_{\mathfrak{g}})$. Also, since
 $
 	N_{\phi_\mathfrak{g}\circ J}(e_i,e_j)=0$, for all $i,j=1,2,3,4$,
  $J$ is an complex structure on $(\mathfrak{g}, [\cdot,\cdot]_\mathfrak{g}, \phi_\mathfrak{g})$. Moreover, we get $\langle({ \phi_\mathfrak{g}}\circ J)(e_i), ({ \phi_\mathfrak{g}}\circ J)(e_j)\rangle=-\langle e_i, e_j\rangle$, for all $i, j=1,2,3,4$. Thus $(\mathfrak{g}, [\cdot, \cdot]_\mathfrak{g}, \phi_{\mathfrak{g}}, J, \langle\cdot,\cdot\rangle)$ is a Norden Hom-Lie algebra.
 \end{example}
\begin{example}
We consider a $4$-dimensional Hom-Lie algebra $(\mathfrak{g},[\cdot,\cdot]_\mathfrak{g},\phi_\mathfrak{g})$ with an arbitrary basis $\{e_1,e_2,e_3,e_4\}$ such that bracket $[\cdot, \cdot]_\mathfrak{g}$ and linear map $\phi_\mathfrak{g}$ on $\mathfrak{g}$ are defined by
\begin{align*}
[e_1,e_4]_\mathfrak{g}=ae_1+ae_2,\ \ \  [e_2,e_3]_\mathfrak{g}=ae_1+ae_2,\ \ \ \ [e_3,e_4]_\mathfrak{g}=-a e_3+ae_4,
\end{align*}
and
\[
\phi_\mathfrak{g}(e_1)=e_2,\ \ \ \phi_\mathfrak{g}(e_2)=e_1,\ \ \ \phi_\mathfrak{g}(e_3)=e_4, \ \ \phi_\mathfrak{g}(e_4)=e_3.
\]
The above bracket is not a Lie bracket on $\mathfrak{g}$ if $a\neq 0$. Now we consider the metric $\langle\cdot,\cdot\rangle$ of $\mathfrak{g}$ as follows:
\begin{equation*}
\begin{bmatrix}
-1&0&0&0\\
0&-1&0&0\\
0&0&1&0\\
0&0&0&1
\end{bmatrix}
.
\end{equation*}
It can be checked easily that  $\langle\phi_\mathfrak{g}(e_i),\phi_\mathfrak{g}(e_j)\rangle=\langle e_i,e_j\rangle$, for all $i,j=1,2,3,4$.
 If isomorphism $J$ is determined as
\[
J(e_1)=e_4,\ \ \ J(e_2)=e_3,\ \ \ J(e_3)=-e_2, \ \ J(e_4)=- e_1,
\]
a simple calculation shows that 
\begin{align*}
(J\circ\phi_\mathfrak{g})e_1=&e_3=(\phi_\mathfrak{g}\circ J)e_1,\ \ \ (J\circ\phi_\mathfrak{g})e_2=e_4=(\phi_\mathfrak{g}\circ J)e_2,\\
(J\circ\phi_\mathfrak{g})e_3=&-e_1=(\phi_\mathfrak{g}\circ J)e_3,\ \ \ (J\circ\phi_\mathfrak{g})e_4=-e_2=(\phi_\mathfrak{g}\circ J)e_4.
\end{align*}
So $J$ is an almost complex structure on $(\mathfrak{g}, [\cdot, \cdot]_\mathfrak{g}, \phi_\mathfrak{g})$.  Also, we see that 
\begin{equation*}
\langle({ \phi_\mathfrak{g}}\circ J)(e_i), ({ \phi_\mathfrak{g}}\circ J)(e_j)\rangle=\langle e_i, e_j\rangle,\ \ \ \forall i, j=1,2,3,4. 
\end{equation*}
Hence $(\mathfrak{g}, [\cdot, \cdot]_\mathfrak{g}, \phi_\mathfrak{g}, J, \langle\cdot,\cdot\rangle)$
 is an almost Norden Hom-Lie algebra. It is easy to see that 
$
N(e_1, e_2)\neq 0. 
$
Thus, $J$ is not integrable and consequently $(\mathfrak{g}, [\cdot, \cdot]_\mathfrak{g}, \phi_\mathfrak{g}, J, \langle\cdot,\cdot\rangle)$ is not a Norden Hom-Lie algebra.
\end{example}
	\begin{example}\label{L8}
	Let  $(\mathfrak{g}, [\cdot,\cdot]_\mathfrak{g}, \phi_\mathfrak{g})$ be a Hom-Lie algebra of the dimensional $4$  with an arbitrary basis $\{e_1,e_2,e_3,e_4\}$ where the bracket
		and linear map $\phi_\mathfrak{g}$ on $\mathfrak{g}$ are given by
		\begin{align*}
		[e_1,e_2]_\mathfrak{g}=&-e_3,\ \ \ [e_1,e_3]_\mathfrak{g}=e_2,\ \ \ [e_2,e_4]_\mathfrak{g}=e_2, \ \ \ [e_3,e_4]_\mathfrak{g}=e_3,
		\end{align*}
		and
		\[
		\phi_\mathfrak{g}(e_1)=e_1,\ \ \ \phi_\mathfrak{g}(e_2)=-e_2,\ \ \ \phi_\mathfrak{g}(e_3)=-e_3, \ \ \phi_\mathfrak{g}(e_4)=e_4.
		\]
		We define the metric $\langle\cdot,\cdot\rangle$ of $\mathfrak{g}$ as follows
		\[
		\begin{bmatrix}
		A&0&0&0\\
		0&B&0&0\\
		0&0&-B&0\\
		0&0&0&-A
		\end{bmatrix}
		.
		\]
		It is easy to see that  $\langle\phi_\mathfrak{g}(e_i),\phi_\mathfrak{g}(e_j)\rangle=\langle e_i,e_j\rangle$, for all $i,j=1,2,3,4$.
		Thus (i) of (\ref{twoo}) holds and so $(\mathfrak{g}, [\cdot, \cdot]_\mathfrak{g}, \phi_{\mathfrak{g}}, \langle\cdot,\cdot\rangle)$ is a pseudo-Riemannian Hom-Lie algebra. Define isomorphism $J$  as
		\[
		J(e_1)=-e_4,\ \ \ J(e_2)=-e_3,\ \ \ J(e_3)=e_2, \ \ J(e_4)= e_1,
		\]
		we deduce that
		\begin{align*}
		(J\circ\phi_\mathfrak{g})e_1=&-e_4=(\phi_\mathfrak{g}\circ J)e_1,\ \ \ (J\circ\phi_\mathfrak{g})e_2=e_3=(\phi_\mathfrak{g}\circ J)e_2,\\
		(J\circ\phi_\mathfrak{g})e_3=&-e_2=(\phi_\mathfrak{g}\circ J)e_3,\ \ \ (J\circ\phi_\mathfrak{g})e_4=e_1=(\phi_\mathfrak{g}\circ J)e_4.\\
		\end{align*}
		Thus $J$ is an almost complex structure on $(\mathfrak{g}, [\cdot, \cdot]_\mathfrak{g}, \phi_{\mathfrak{g}})$. Also, we obtain
		\begin{align*}
		N_{\phi_\mathfrak{g}\circ J}(e_i,e_j)=0,\ \ \ \ \forall i,j=1,2,3,4,
		\end{align*}
		i.e., $J$ is an complex structure on $(\mathfrak{g}, [\cdot,\cdot]_\mathfrak{g}, \phi_\mathfrak{g})$. It is easy to check that $\langle({ \phi_\mathfrak{g}}\circ J)(e_i), ({ \phi_\mathfrak{g}}\circ J)(e_j)\rangle=-\langle e_i, e_j\rangle$, for all $i, j=1,2,3,4$, and so $(\mathfrak{g}, [\cdot, \cdot]_\mathfrak{g}, \phi_{\mathfrak{g}}, J, \langle\cdot,\cdot\rangle)$ is a Norden Hom-Lie algebra.
\end{example}
\begin{theorem}\label{MA8}
	Let $( J, \langle\cdot,\cdot\rangle)$ be a left-invariant almost Norden structure on a Hom-Lie group $(G, \diamond, e_{\Phi}, \Phi)$ (or, an  almost Norden structure on a Hom-Lie algebra  $(\mathfrak{g^!}, [\cdot, \cdot]_{\mathfrak{g^!}}, \phi_\mathfrak{g^!})$). Then we have 
	\begin{align*}
 \langle{ \phi_\mathfrak{g^!}}(y),(	\nabla_x({ \phi_\mathfrak{g^!}}\circ J))z\rangle
=\langle (	\nabla_x({ \phi_\mathfrak{g^!}}\circ J))y,{ \phi_\mathfrak{g^!}}(z)\rangle,
	\end{align*}
	for any $x,y,z\in \mathfrak{g^!}$,  where $\nabla$ is the left-invariant Hom-Levi-Civita connection on $G$ (or, the Hom-Levi-Civita connection on $\mathfrak{g^!}$) and $(\nabla_x( { \phi_\mathfrak{g^!}}\circ J))y=\nabla_x( { \phi_\mathfrak{g^!}}\circ J)y-(
	{ \phi_\mathfrak{g^!}}\circ J)\nabla_xy$.
\end{theorem}
\begin{proof}
From (ii) of (\ref{twoo}), we have
		\begin{align*}
	&\langle \nabla_x ({ \phi_\mathfrak{g^!}}\circ J)y,{ \phi_\mathfrak{g^!}}(z)\rangle=-\langle{ \phi_\mathfrak{g^!}}(({ \phi_\mathfrak{g^!}}\circ J)y),\nabla_x z\rangle,\\
	&\langle \nabla_x y,{ \phi_\mathfrak{g^!}}(({ \phi_\mathfrak{g^!}}\circ J)z)\rangle=-\langle{ \phi_\mathfrak{g^!}}(y),\nabla_x ({ \phi_\mathfrak{g^!}}\circ J)z\rangle.
		\end{align*}
			Since $\langle ({ \phi_\mathfrak{g^!}}\circ J)\cdot,\cdot \rangle=\langle \cdot,({ \phi_\mathfrak{g^!}}\circ J)\cdot \rangle$,  from the above equations it follows that
				\begin{align*}
			&\langle \nabla_x ({ \phi_\mathfrak{g^!}}\circ J)y,{ \phi_\mathfrak{g^!}}(z)\rangle=-\langle{ \phi_\mathfrak{g^!}}(y),({ \phi_\mathfrak{g^!}}\circ J)(\nabla_x z)\rangle,\\
			&\langle ({ \phi_\mathfrak{g^!}}\circ J)(\nabla_x y),{ \phi_\mathfrak{g^!}}(z)\rangle=-\langle{ \phi_\mathfrak{g^!}}(y),\nabla_x ({ \phi_\mathfrak{g^!}}\circ J)z\rangle.
			\end{align*}
		 Subtracting the above equations,  we conclude  the assertion. 
		\end{proof}
%%%%%%%%%%%%%%%%%%%%%%%%%%%%%%%%%%%%%%%%%%%%%%%%%%%%%%%%%
\subsection{Holomorphic tensors on Hom-Lie groups and Hom-Lie algebras}
 We present   the theory of Tachibana operators  for Hom-Lie groups and Hom-Lie algebras  and then 
give the notion of a holomorphic tensor on a Hom-Lie group and a Hom-Lie algebra.

A left-invariant complex Hom-Lie group (or, a complex  Hom-Lie algebra) is a left-invariant almost complex Hom-Lie group $(G, \diamond, e_\Phi, \Phi, J)$ (or, an almost complex Hom-Lie algebra $(\mathfrak{g^!}, [\cdot, \cdot]_{\mathfrak{g^!}}, \phi_\mathfrak{g^!}, J)$) such that the tensor $J$ is integrable.
In this case, the left-invariant almost complex structure (or, the almost complex structure) $J$ is called a left-invariant complex structure (or, a complex structure).
On the other hand, in order that an almost complex structure be integrable, it is necessary and sufficient that  for any $x,y\in \mathfrak{g^!}$,  we consider a left-invariant connection   $\nabla$ 
 such that  the  symmetric condition be held, i.e., 
\begin{align*}
[x,y]_{\mathfrak{g^!}}&=\nabla_x y-\nabla_y x, 
\end{align*}
and ${ \phi_\mathfrak{g^!}}\circ J$ is invariant with respect to $\nabla$, that is $\nabla
({ \phi_\mathfrak{g^!}}\circ J)=0$ or 
\[\nabla_x 
({ \phi_\mathfrak{g^!}}\circ J)=({ \phi_\mathfrak{g^!}}\circ J)\nabla_x.
\]
Note that condition $\nabla_x({ \phi_\mathfrak{g^!}}\circ J)(y)=({ \phi_\mathfrak{g^!}}\circ J)\nabla_x y$ is equivalent with 
\begin{equation*}
\nabla_{({ \phi_\mathfrak{g^!}}\circ J)(x)}({ \phi_\mathfrak{g^!}}\circ J)(y)=({ \phi_\mathfrak{g^!}}\circ J)\nabla_{({ \phi_\mathfrak{g^!}}\circ J)(x)} y,
\end{equation*}
and
\begin{equation*}
\nabla_x y=-({ \phi_\mathfrak{g^!}}\circ J)\nabla_x({ \phi_\mathfrak{g^!}}\circ J)(y).
\end{equation*}
 (\ref{M1}) and two last equations imply 
\begin{align}
N_{{ \phi_\mathfrak{g^!}}\circ J}(x,y)=&\nabla_{({ \phi_\mathfrak{g^!}}\circ J)(x)}  ({ \phi_\mathfrak{g^!}}\circ J)(y)-\nabla_{({ \phi_\mathfrak{g^!}}\circ J)(y)} ({ \phi_\mathfrak{g^!}}\circ J)(x)-({ \phi_\mathfrak{g^!}}\circ J)(\nabla_{({ \phi_\mathfrak{g^!}}\circ J)(x)} y\label{AM11}\\
&-\nabla_y ({ \phi_\mathfrak{g^!}}\circ J)(x))-({ \phi_\mathfrak{g^!}}\circ J)(\nabla_x({ \phi_\mathfrak{g^!}}\circ J)(y)-\nabla_{({ \phi_\mathfrak{g^!}}\circ J)(y)} x)-\nabla_x y-\nabla_yx=0\nonumber.
\end{align}
It is also known that the integrability of $J$ is equivalent to the vanishing of the Nijenhuis tensor $N_{{ \phi_\mathfrak{g^!}}\circ J}$. So, $J$  is integrable.

We can give another equivalent definition  of left-invariant complex Hom-Lie group $(G, \diamond, e_\Phi, \Phi, J)$ (or, complex  Hom-Lie algebras $(\mathfrak{g^!}, [\cdot, \cdot]_{\mathfrak{g^!}}, \phi_\mathfrak{g^!}, J)$) in terms of  holomorphics. Given a 
tensor  $\omega\in \mathcal{T}_{q}^{0}(\mathfrak{g^!})$, the purity means that for any $x_1,\cdots,x_q\in \mathfrak{g^!}$ the following condition should be held
%A tensor field $\omega$ of type $(0, q)$ is called a pure tensor field with respect to $J$ if
\[
\omega(({ \phi_\mathfrak{g^!}}\circ J)x_1,x_2, \cdots, x_q)=\omega(x_1,({ \phi_\mathfrak{g^!}}\circ J)x_2, \cdots, x_q)=\cdots=\omega(x_1,x_2, \cdots, ({ \phi_\mathfrak{g^!}}\circ J)x_q).
\]
We define an operator
\[
\Phi_{{{ \phi_\mathfrak{g^!}}\circ J}}:\mathcal{T}_{q}^{0}(\mathfrak{g^!})\longrightarrow \mathcal{T}_{q+1}^{0}(\mathfrak{g^!}),
\]
applied to the pure tensor $\omega\in \mathcal{T}_{q}^{0}(\mathfrak{g^!})$ by
\begin{align}\label{L3}
(\Phi_{{{ \phi_\mathfrak{g^!}}\circ J}}\ \omega)(x, y_1,y_2,\cdots, y_q)&=
\sum_{i=1}^q \omega( \phi_\mathfrak{g^!}(y_1),\cdots,\phi_\mathfrak{g^!}(y_{i-1}),[{y_i},({ \phi_\mathfrak{g^!}}\circ J)x]_{\mathfrak{g^!}}\\
&\ \ \ -({ \phi_\mathfrak{g^!}}\circ J) [{y_i},x]_{\mathfrak{g^!}},\phi_\mathfrak{g^!}(y_{i+1}),\cdots, \phi_\mathfrak{g^!}(y_q)),\nonumber
%&\omega((\pounds_{v_1}({ \phi_\mathfrak{g}}\circ J))u,\phi_\mathfrak{g}(v_2),\cdots, \phi_\mathfrak{g}(v_q))\label{L3}\\
%&+\cdots+\omega(\phi_\mathfrak{g}(v_1),\phi_\mathfrak{g}(v_2),\cdots,(\pounds_{v_q}({ \phi_\mathfrak{g}}\circ J))u ),\nonumber
\end{align}
%where
%\begin{align*}
%ad_{v_1}{ \phi_\mathfrak{g}}\circ J)(u)-({ \phi_\mathfrak{g}}\circ J) ad _{v_1}u,
%\end{align*}
for any $x,y_1,\cdots,y_q\in \mathfrak{g^!}$.
The  pure tensor $\omega$ is called holomorphic if $J$  is a complex structure on $\mathfrak{g^!}$ and 
\[
\Phi_{{{ \phi_\mathfrak{g^!}}\circ J}}\ \omega=0.
\]
According to Definition \ref{L100}, we see that the Norden metric $\langle\cdot,\cdot\rangle$ on the almost complex Hom-Lie algebra  $(\mathfrak{g^!}, [\cdot, \cdot]_{\mathfrak{g^!}}, \phi_{\mathfrak{g^!}},J)$
 is a pure tensor with respect to $J$.
%%%%%%%%%%%%%%%%%%%%%%%%%%%%%%%%%%%%%%%%%%%%%%%%%%%%%%%%
\section{ Left-invariant K\"{a}hler-Norden Hom-Lie groups (K\"{a}hler-Norden Hom-Lie algebras)}
In this section, we introduce left-invariant K\"{a}hler-Norden structures on Hom-Lie groups (K\"{a}hler-Norden structures on Hom-Lie algebras) and we present  examples of these structures.
\begin{definition}\label{KKKKK}
A left-invariant K\"{a}hler Hom-Lie group (or, a K\"{a}hler-Norden Hom-Lie algebra) is a left-invariant almost Norden Hom-Lie group $(G, \diamond, e_\Phi, \Phi, J, \langle\cdot, \cdot\rangle)$ (or, an almost Norden Hom-Lie algebra $(\mathfrak{g^!}, [\cdot, \cdot]_{\mathfrak{g^!}}, \phi_{\mathfrak{g^!}},J,\langle\cdot,\cdot\rangle)$)  such that ${ \phi_\mathfrak{g^!}}\circ J$ is invariant with respect to the Hom-Levi-Civita connection $\nabla$, i.e.,
\begin{align}\label{L19}
\nabla( { \phi_\mathfrak{g^!}}\circ J)=0.
 \end{align}
 % for any $x\in \mathfrak{g^!}$.
\end{definition}
Note that 
(\ref{AM11}) implies that the structure $J$ introduced in Definition \ref{KKKKK} is integrable.
%66666666666666666666666666666666666666666666666666666666666666666666
\begin{remark}
The above definition is hold for an arbitrary Hom-Lie algebra $(g, [\cdot, \cdot], \phi)$.
\end{remark}
\begin{example}\label{L98}
Consider the Hom-Lie algebra $4$-dimensional $(\mathfrak{g}, [\cdot,\cdot]_\mathfrak{g}, \phi_\mathfrak{g})$ introduced in Example (\ref{L8}). We study K\"{a}hlerian-Norden property for this Hom-Lie algebra. If we denote the Hom-Levi-Civita connection by $\nabla$, from  Koszul's formula given by (\ref{Koszul}) we deduce that $\langle \nabla_{e_i} e_j, \phi_\mathfrak{g}(e_k)\rangle=0$, for all $i,j,k=1, 2, 3, 4$, expect
\begin{align*}
&\nabla_{e_2} e_1=e_3,\ \ \ \nabla_{e_2} e_2=-\frac{B}{A}e_4, \ \ \nabla_{e_2} e_3=-\frac{B}{A}e_1,\ \  \nabla_{e_2} e_4=e_2,\\
&\nabla_{e_3} e_1=-e_2,\ \ \nabla_{e_3} e_2=-\frac{B}{A}e_1, \ \  \nabla_{e_3} e_3=\frac{B}{A}e_4,\ \  \nabla_{e_3} e_4=e_3.
\end{align*}
Easily we can see that 
 ${ \phi_\mathfrak{g}}\circ J$ is invariant with respect to the Hom-Levi-Civita connection computed in the above, i.e, (\ref{L19}) holds. So  $(\mathfrak{g}, [\cdot, \cdot]_\mathfrak{g}, \phi_{\mathfrak{g}},J,\langle\cdot,\cdot\rangle)$ is a  K\"{a}hler-Norden Hom-Lie algebra.
\end{example}
\begin{example}\label{L91}
	Let $\{e_1,\cdots, e_6\}$ be a basis of a $6$-dimensional Hom-Lie algebra $(\mathfrak{g},[\cdot, \cdot]_\mathfrak{g}, \phi_{\mathfrak{g}})$, where
		\begin{align*}
	[e_3,e_5]_\mathfrak{g}=&-e_2,\ \ \ [e_3,e_6]_\mathfrak{g}=e_1,\ \ \ [e_4,e_5]_\mathfrak{g}=e_1, \ \ \ [e_4,e_6]_\mathfrak{g}=e_2,
	 \ \ \ [e_5,e_6]_\mathfrak{g}=e_3,
	\end{align*}
	and
	\[
	\phi_\mathfrak{g}(e_1)=-e_1,\ \ \ \phi_\mathfrak{g}(e_2)=-e_2,\ \ \ \phi_\mathfrak{g}(e_3)=e_3, \ \ \phi_\mathfrak{g}(e_4)=e_4, \ \ \phi_\mathfrak{g}(e_5)=-e_5, \ \ \phi_\mathfrak{g}(e_6)=-e_6.
	\]
We define an isomorphism $J$  on $\mathfrak{g}$ as follows
	\[
	J(e_1)=-e_2,\ \ \ J(e_2)=e_1,\ \ \ J(e_3)=e_4, \ \ J(e_4)= -e_3,
	 \ \ J(e_5)= e_6, \ \ J(e_6)= -e_5,
	 	\]
which implies that
	\begin{align*}
	(J\circ\phi_\mathfrak{g})e_1=&e_2=(\phi_\mathfrak{g}\circ J)e_1,\ \ \ \ \ (J\circ\phi_\mathfrak{g})e_2=-e_1=(\phi_\mathfrak{g}\circ J)e_2,\\
	(J\circ\phi_\mathfrak{g})e_3=&e_4=(\phi_\mathfrak{g}\circ J)e_3,\ \ \ \ \ (J\circ\phi_\mathfrak{g})e_4=-e_3=(\phi_\mathfrak{g}\circ J)e_4,\\
	(J\circ\phi_\mathfrak{g})e_5=&-e_6=(\phi_\mathfrak{g}\circ J)e_5,\ \  (J\circ\phi_\mathfrak{g})e_6=e_5=(\phi_\mathfrak{g}\circ J)e_6.
	\end{align*}
	Thus $J$ is an almost complex structure on $(\mathfrak{g}, [\cdot, \cdot]_\mathfrak{g}, \phi_{\mathfrak{g}})$. On the other hand, $N_{\phi_\mathfrak{g}\circ J}(e_i,e_j)=0$, for all  $i,j=1,\cdots,6$, so $J$ is a complex structure.
	If	we consider the metric $\langle\cdot,\cdot\rangle$ of $\mathfrak{g}$ as follows
	\[
	\begin{bmatrix}
	0&0&0&0&\frac{A}{2}&0\\
	0&0&0&0&0&\frac{A}{2}\\
	0&0&A&0&0&0\\
	0&0&0&-A&0&0\\
		\frac{A}{2}&0&0&0&0&0\\
	0&\frac{A}{2}&0&0&0&0
	\end{bmatrix}
	,\]
then $\langle\phi_\mathfrak{g}(e_i),\phi_\mathfrak{g}(e_j)\rangle=\langle e_i,e_j\rangle$, for all $i,j=1,\cdots,6$.
So  $(\mathfrak{g}, [\cdot, \cdot]_\mathfrak{g}, \phi_{\mathfrak{g}}, \langle\cdot,\cdot\rangle)$ is a pseudo-Riemannian Hom-Lie algebra. Also, since $\langle({ \phi_\mathfrak{g}}\circ J)(e_i), ({ \phi_\mathfrak{g}}\circ J)(e_j)\rangle=-\langle e_i, e_j\rangle$, for all $i,j=1,\cdots,6$, thus $(\mathfrak{g}, [\cdot, \cdot]_\mathfrak{g}, \phi_{\mathfrak{g}}, J, \langle\cdot,\cdot\rangle)$ is a Norden Hom-Lie algebra. 
Using (\ref{Koszul}), we obtain $\langle e_i\cdot e_j, \phi_\mathfrak{g}(e_k)\rangle=0$, for all $i,j,k=1,\cdots,6$, expect
\begin{align}
&%e_3\cdot e_5=0,\ \ \ \ e_3\cdot e_6=0,\ \ \ \ \ \ \ e_4\cdot e_5=0, \ \ \ \ \  e_4\cdot e_6=0,
\nabla_{e_5} e_3=e_2,\ \ \ \ \ \ \ \  \nabla_{e_5} e_4=-e_1,\ \ \ \ \ \ \ \nabla_{e_5} e_5=\frac{1}{2}e_4, \ \ \ \ \ \ \
\nabla_{e_5} e_6=\frac{1}{2}e_3, \label{L93}\\
&\nabla_{e_6} e_3=-e_1, \ \ \ \ \ \ \nabla{ e_6} e_4=-e_2,\ \ \ \ \ \ \   \nabla_{e_6} e_5=-\frac{1}{2}e_3,\ \ \ \ \ \nabla_{ e_6} e_6=\frac{1}{2}e_4,\nonumber
\end{align}
where $\nabla$ denote the Hom-Levi-Civita connection on $\mathfrak{g}$.
Easily we can see that 
 (\ref{L19}) holds, therefore  $(\mathfrak{g}, [\cdot, \cdot]_\mathfrak{g}, \phi_{\mathfrak{g}},J,\langle\cdot,\cdot\rangle)$ is a  K\"{a}hler-Norden Hom-Lie algebra.
	\end{example}
\begin{proposition}
Let $(G, \diamond, e_\Phi, \Phi, J, \langle\cdot, \cdot\rangle)$ be a left-invariant almost Norden Hom-Lie group (or, $(\mathfrak{g}, [\cdot, \cdot]_{\mathfrak{g^!}}, \phi_{\mathfrak{g^!}},J,\langle\cdot,\cdot\rangle)$ be an almost Norden Hom-Lie algebra). Then $(G, \diamond, e_\Phi, \Phi, J, \langle\cdot, \cdot\rangle)$ (or,\\
 $(\mathfrak{g^!}, [\cdot, \cdot]_{\mathfrak{g^!}}, \phi_{\mathfrak{g^!}},J,\langle\cdot,\cdot\rangle)$) is a left-invariant K\"{a}hler-Norden Hom-Lie group (or, a K\"{a}hler-Norden Hom-Lie algebra) if and only if 
\begin{align}\label{L13}
(\nabla_{({ \phi_\mathfrak{g^!}}\circ J)x} ({ \phi_\mathfrak{g^!}}\circ J))y=-({ \phi_\mathfrak{g^!}}\circ J)(\nabla_x ({ \phi_\mathfrak{g^!}}\circ J))y,
\end{align}
for any $x,y\in\mathfrak{g^!}$.
\end{proposition}
\begin{proof}
	Consider $\alpha(x,y,z)=\langle(\nabla_x({ \phi_\mathfrak{g^!}}\circ J))y,\phi_{\mathfrak{g^!}}(z)\rangle$. Theorem \ref{MA8} implies that $\alpha$ is symmetric in the last two variables
	\begin{align}\label{L15}
\alpha(x,y,z)=\alpha(x,z,y).
	\end{align}
	Using the definition of $\alpha$, we have
	\begin{align*}
\alpha(({ \phi_\mathfrak{g^!}}\circ J)x,y,z)=
\langle(\nabla_{({ \phi_\mathfrak{g^!}}\circ J)x}({ \phi_\mathfrak{g^!}}\circ J))y,\phi_{\mathfrak{g^!}}(z)\rangle.
	\end{align*}
	(\ref{L14}),  (\ref{L13}) and the above equation imply
		\begin{align*}
	\alpha(({ \phi_\mathfrak{g^!}}\circ J)x,y,z)=&-
	\langle({ \phi_\mathfrak{g^!}}\circ J)(\nabla_{x}( { \phi_\mathfrak{g^!}}\circ J))y,\phi_{\mathfrak{g^!}}(z)\rangle=-
	\langle(\nabla_{x}( { \phi_\mathfrak{g^!}}\circ J))y,({ \phi_\mathfrak{g^!}}\circ J)(\phi_{\mathfrak{g^!}}(z))\rangle,
	\end{align*}
	from which it follows 
		\begin{align}\label{L16}
	\alpha(({ \phi_\mathfrak{g^!}}\circ J)x,y,z)=-\alpha(x,y,({ \phi_\mathfrak{g^!}}\circ J)z).
	\end{align}
	On the other hand, we have
		\begin{align*}
	\alpha(({ \phi_\mathfrak{g^!}}\circ J)x,y,z)=-
	\langle({ \phi_\mathfrak{g^!}}\circ J)(\nabla_x ({ \phi_\mathfrak{g^!}}\circ J))y,\phi_{\mathfrak{g^!}}(z)\rangle.
	\end{align*} 
The above equation  is equivalent to  
		\begin{align*}
	\alpha(({ \phi_\mathfrak{g^!}}\circ J)x,y,z)=
	\langle(\nabla_x ({ \phi_\mathfrak{g^!}}\circ J))({ \phi_\mathfrak{g^!}}\circ J)y,\phi_{\mathfrak{g^!}}(z)\rangle,
	\end{align*}
	which gives us
		\begin{align}\label{L17}
	\alpha(({ \phi_\mathfrak{g^!}}\circ J)x,y,z)=
	\alpha(x,({ \phi_\mathfrak{g^!}}\circ J)y,z).
	\end{align}
	By (\ref{L15}), (\ref{L16}) and (\ref{L17}), we get
		\begin{align*}
	\alpha(({ \phi_\mathfrak{g^!}}\circ J)x,y,z)=
	\alpha(({ \phi_\mathfrak{g^!}}\circ J)x,z,y)=
		\alpha(x,({ \phi_\mathfrak{g^!}}\circ J)z,y)=
		\alpha(x,y,({ \phi_\mathfrak{g^!}}\circ J)z)=
		-\alpha(({ \phi_\mathfrak{g^!}}\circ J)x,y,z).
	\end{align*}
From the above equation and the non-degenerate condition of $\langle\cdot ,\cdot \rangle$, we have  $\nabla({ \phi_\mathfrak{g^!}}\circ J)=0$.
	\end{proof}
%%%%%%%%%%%%%%%%%%%%%%%%%%%%%%%%%%%%%%%%%%%%%%%%%%%%%%%%%%%%%%%%%%%%%%%%%%%
\subsection{Left-invariant abelian structures on Hom-Lie groups (abelian complex structures on Hom-Lie algebras)}
Let $(G, \diamond, e_\Phi, \Phi, J)$ be a left-invariant almost complex Hom-Lie group (or, $(\mathfrak{g^!}, [\cdot  , \cdot ]_{\mathfrak{g^!}}, {\phi_\mathfrak{g^!}}, J)$ be an almost complex Hom-Lie algebra). 
The left-invariant almost complex $J$ (or, the almost complex structure $J$) is called abelian if 
\begin{align}\label{L22}
[({ \phi_\mathfrak{g^!}}\circ J)x,({ \phi_\mathfrak{g^!}}\circ J)y]_{\mathfrak{g^!}}=[x,y]_{\mathfrak{g^!}},
\end{align}
for any $x,y\in\mathfrak{g^!}$.
\begin{lemma}
A left-invariant almost Norden Hom-Lie group $(G, \diamond, e_\Phi, \Phi, J, \langle\cdot, \cdot\rangle)$ (or, an  almost Norden Hom-Lie algebra $(\mathfrak{g^!}, [\cdot, \cdot]_{\mathfrak{g^!}}, \phi_{\mathfrak{g^!}},J,\langle\cdot,\cdot\rangle)$) is a  left-invariant K\"{a}hler-Norden Hom-Lie group (or, a K\"{a}hler-Norden Hom-Lie algebra) if for any $x\in\mathfrak{g^!}$, the following condition  satisfies
\begin{align}\label{L18}
\nabla_{({ \phi_\mathfrak{g^!}}\circ J)x}=-({ \phi_\mathfrak{g^!}}\circ J) \nabla_x.
\end{align}
Moreover, the complex structure $J$ is abelian.
\end{lemma}
\begin{proof}
	From (\ref{L18}), it follows
	\begin{align*}
&	\nabla_{({ \phi_\mathfrak{g^!}}\circ J)x} ({ \phi_\mathfrak{g^!}}\circ J)=-({ \phi_\mathfrak{g^!}}\circ J)\nabla_x( { \phi_\mathfrak{g^!}}\circ J),
\end{align*}
and
	\begin{align*}
	&({ \phi_\mathfrak{g^!}}\circ J)\nabla_{({ \phi_\mathfrak{g^!}}\circ J)x}= \nabla_x.
	\end{align*}
	Subtracting the above two equations, we get (\ref{L13}). Now, let $x,y\in\mathfrak{g^!}$. Using (i) of (\ref{twoo}), we obtain
	\begin{align*}
[({ \phi_\mathfrak{g^!}}\circ J)x,({ \phi_\mathfrak{g^!}}\circ J)y]_{\mathfrak{g^!}}=
\nabla_{({ \phi_\mathfrak{g^!}}\circ J)x}({ \phi_\mathfrak{g^!}}\circ J)y-
\nabla_{({ \phi_\mathfrak{g^!}}\circ J)y}({ \phi_\mathfrak{g^!}}\circ J)x.
	\end{align*}
	Applying (\ref{L19}) and (\ref{L18}) in the above equation yields 
	\begin{align*}
	[({ \phi_\mathfrak{g^!}}\circ J)x,({ \phi_\mathfrak{g^!}}\circ J)y]_{\mathfrak{g^!}}=
-	({ \phi_\mathfrak{g^!}}\circ J)(\nabla_x({ \phi_\mathfrak{g^!}}\circ J)y)+
	({ \phi_\mathfrak{g^!}}\circ J)(\nabla_y({ \phi_\mathfrak{g^!}}\circ J)x)=\nabla_x y-\nabla_y x=[x,y]_{\mathfrak{g^!}}.
	\end{align*}
	\end{proof}
\begin{proposition}\label{L29}
	Let  $(G, \diamond, e_\Phi, \Phi, J, \langle\cdot, \cdot\rangle)$ be a left-invariant K\"{a}hler-Norden Hom-Lie group (or, $(\mathfrak{g^!}, [\cdot, \cdot]_{\mathfrak{g^!}}, \phi_{\mathfrak{g^!}},J,\langle\cdot,\cdot\rangle)$ be a  K\"{a}hler-Norden Hom-Lie algebra). If the left-invariant complex structure $J$ (or, the complex structure $J$) is abelian, then 
	(\ref{L18}) holds. 
\end{proposition}
\begin{proof}
For any $x,y,z\in\mathfrak{g^!}$,  setting $\alpha(x,y,z)=\langle (\nabla_{({ \phi_\mathfrak{g^!}}\circ J)x}y+({ \phi_\mathfrak{g^!}}\circ J)\nabla_xy, \phi_\mathfrak{g^!}(z)\rangle$ and using (ii) of (\ref{twoo}), we get
\begin{align*}
\alpha(x,y,z)=-\langle \nabla_{({ \phi_\mathfrak{g^!}}\circ J)x} z,\phi_\mathfrak{g^!}(y)\rangle+\langle \nabla_x y,(\phi_\mathfrak{g^!}\circ J) \phi_\mathfrak{g^!}(z)\rangle=-\langle \nabla_{({ \phi_\mathfrak{g^!}}\circ J)x} z+ \nabla_x (\phi_\mathfrak{g^!}\circ J)z, \phi_\mathfrak{g^!}(y)\rangle.
\end{align*}
According to (\ref{L19}), the above equation gives 
\begin{align*}
\alpha(x,y,z)=-\langle \nabla_{({ \phi_\mathfrak{g^!}}\circ J)x} z+ (\phi_\mathfrak{g^!}\circ J)(\nabla_x z), \phi_\mathfrak{g^!}(y)\rangle,
\end{align*}
from which it follows
\begin{align}\label{L24}
\alpha(x,y,z)=-\alpha(x,z,y).
\end{align}
As $J$ is abelian, from (i) of (\ref{twoo}) and (\ref{L22}), we have
\[
(\nabla_{({ \phi_\mathfrak{g^!}}\circ J)x}+({ \phi_\mathfrak{g^!}}\circ J) \nabla_x)y=
(\nabla_{({ \phi_\mathfrak{g^!}}\circ J)y}+({ \phi_\mathfrak{g^!}}\circ J)\nabla_y)x,
\]
hence $\alpha$ is symmetric in the first two variable
\begin{align}\label{L26}
\alpha(x,y,z)=\alpha(y,x,z).
\end{align}
By (\ref{L24}) and (\ref{L26}), we conclude that 
\begin{align*}
\alpha(x,y,z)=-\alpha(x,z,y)=-\alpha(z,x,y)=\alpha(z,y,x).
\end{align*}
The last equation and (\ref{L24}) imply
\begin{align*}
\alpha(x,y,z)=\alpha(y,x,z)=\alpha(z,x,y)=\alpha(x,z,y).
\end{align*}
(\ref{L24}) and the above equation deduce that $\alpha=0$,  hence the proof completes. 
	\end{proof}
\begin{proposition}\label{L30}
	Assume that $(G, \diamond, e_\Phi, \Phi, J, \langle\cdot, \cdot\rangle)$ is a left-invariant K\"{a}hler-Norden Hom-Lie group (or, $(\mathfrak{g^!}, [\cdot, \cdot]_{\mathfrak{g^!}}, \phi_{\mathfrak{g^!}},J,\langle\cdot,\cdot\rangle)$ is a  K\"{a}hler-Norden Hom-Lie algebra) such that the left-invariant complex structure $J$ (or, the complex structure $J$) is abelian. Then for the left-invariant Hom-Levi-Civita connection $\nabla$ (or, the Hom-Levi-Civita connection $\nabla$), we have 
	\begin{align}\label{L92}
	2\nabla_xy=[x,y]_{\mathfrak{g^!}}-({ \phi_\mathfrak{g^!}}\circ J)[x,({ \phi_\mathfrak{g^!}}\circ J)y]_{\mathfrak{g^!}},
	\end{align} 
	for any $x,y\in \mathfrak{g^!}$.
\end{proposition}
\begin{proof}
Using (i) of (\ref{twoo}) and (\ref{L19}), we have 
	\begin{align*}
	[({ \phi_\mathfrak{g^!}}\circ J)x,y]_{\mathfrak{g^!}}-({ \phi_\mathfrak{g^!}}\circ J)[x,y]_{\mathfrak{g^!}}=&
\nabla_{	({ \phi_\mathfrak{g^!}}\circ J)x} y-\nabla_y ({ \phi_\mathfrak{g^!}}\circ J)x-({ \phi_\mathfrak{g^!}}\circ J)(\nabla_x y-\nabla_y x)\\
	=&\nabla_{({ \phi_\mathfrak{g^!}}\circ J)x} y-({ \phi_\mathfrak{g^!}}\circ J)(\nabla_x y).
	\end{align*}
	Applying (\ref{L29}) and $({ \phi_\mathfrak{g}}\circ J)^2=-Id_{\mathfrak{g^!}}$ in the above equation, we conclude the assertion. 
	\end{proof}
\begin{example}
We consider the  K\"{a}hler-Norden Hom-Lie algebra $(\mathfrak{g}, [\cdot, \cdot]_\mathfrak{g}, \phi_{\mathfrak{g}},J,\langle\cdot,\cdot\rangle)$ presented in Example 	(\ref{L98}).
As
\begin{align*}
&[({ \phi_\mathfrak{g}}\circ J)e_1,({ \phi_\mathfrak{g}}\circ J)e_2]_\mathfrak{g}=e_3\neq[e_1,e_2]_\mathfrak{g},\ \ \ \ \ \ \ \ \ [({ \phi_\mathfrak{g}}\circ J)e_1,({ \phi_\mathfrak{g}}\circ J)e_3]_\mathfrak{g}=-e_2\neq[e_1,e_3]_\mathfrak{g},\\
& [({ \phi_\mathfrak{g}}\circ J)e_2,({ \phi_\mathfrak{g}}\circ J)e_4]_\mathfrak{g}=-e_2\neq[e_2,e_4]_\mathfrak{g}, \ \ \ \ \ \ \  [({ \phi_\mathfrak{g}}\circ J)e_3,({ \phi_\mathfrak{g}}\circ J)e_4]_\mathfrak{g}=-e_3\neq[e_3,e_4]_\mathfrak{g},
\end{align*}
So the complex structure $J$ is not abelian.
\end{example}
\begin{example}\label{L94}
For the  K\"{a}hler-Norden Hom-Lie algebra $(\mathfrak{g}, [\cdot, \cdot]_\mathfrak{g}, \phi_{\mathfrak{g}},J,\langle\cdot,\cdot\rangle)$ in Example (\ref{L91}), we have
\begin{align*}
&[({ \phi_\mathfrak{g}}\circ J)e_3,({ \phi_\mathfrak{g}}\circ J)e_5]_\mathfrak{g}=-e_2=[e_3,e_5]_\mathfrak{g},\ \ \ \ \ [({ \phi_\mathfrak{g}}\circ J)e_3,({ \phi_\mathfrak{g}}\circ J)e_6]_\mathfrak{g}=e_1=[e_3,e_6]_\mathfrak{g},\\
& [({ \phi_\mathfrak{g}}\circ J)e_4,({ \phi_\mathfrak{g}}\circ J)e_5]_\mathfrak{g}=e_1=[e_4,e_5]_\mathfrak{g}, \ \ \ \ \ \ \  [({ \phi_\mathfrak{g}}\circ J)e_4,({ \phi_\mathfrak{g}}\circ J)e_6]_\mathfrak{g}=e_2=[e_4,e_6]_\mathfrak{g},\\
 &
\ \ \ \ \ \ \ \ \ \ \ \ \ \ \ \ \ \ \ \ \ \ \ \ \ [({ \phi_\mathfrak{g}}\circ J)e_5,({ \phi_\mathfrak{g}}\circ J)e_6]_\mathfrak{g}=e_3=[e_5,e_6]_\mathfrak{g},
\end{align*}
which give that  $J$ satisfies  in  (\ref{L22}), i.e.,  the complex structure $J$ is abelian. We get that $[e_i,e_j]_\mathfrak{g}-({ \phi_\mathfrak{g}}\circ J)[e_i,({ \phi_\mathfrak{g}}\circ J)e_j]_\mathfrak{g}=0$, for all $i,j=1,\cdots,6$, expect 
	\begin{align*}
&[e_5,e_3]_\mathfrak{g}-({ \phi_\mathfrak{g}}\circ J)[e_5,({ \phi_\mathfrak{g}}\circ J)e_3]_\mathfrak{g}=2e_2,\ \ \ \ \ \ \  [e_5,e_4]_\mathfrak{g}-({ \phi_\mathfrak{g}}\circ J)[e_5,({ \phi_\mathfrak{g}}\circ J)e_4]_\mathfrak{g}=-2e_1,\\
&[e_5,e_5]_\mathfrak{g}-({ \phi_\mathfrak{g}}\circ J)[e_5,({ \phi_\mathfrak{g}}\circ J)e_5]_\mathfrak{g}=e_4,\ \ \ \ \ \ \ \ \  [e_5,e_6]_\mathfrak{g}-({ \phi_\mathfrak{g}}\circ J)[e_5,({ \phi_\mathfrak{g}}\circ J)e_6]_\mathfrak{g}=e_3,\\
&[e_6,e_3]_\mathfrak{g}-({ \phi_\mathfrak{g}}\circ J)[e_6,({ \phi_\mathfrak{g}}\circ J)e_3]_\mathfrak{g}=-2e_1,\ \ \ \ \  [e_6,e_4]_\mathfrak{g}-({ \phi_\mathfrak{g}}\circ J)[e_6,({ \phi_\mathfrak{g}}\circ J)e_4]_\mathfrak{g}=-2e_2,\\
&[e_6,e_5]_\mathfrak{g}-({ \phi_\mathfrak{g}}\circ J)[e_6,({ \phi_\mathfrak{g}}\circ J)e_5]_\mathfrak{g}=-e_3,\ \ \ \ \ \ \  [e_6,e_6]_\mathfrak{g}-({ \phi_\mathfrak{g}}\circ J)[e_6,({ \phi_\mathfrak{g}}\circ J)e_6]_\mathfrak{g}=e_4.
\end{align*} 
 Applying  (\ref{L93}) and the above equations, we conclude that (\ref{L92}) holds.
\end{example}
%%%%%%%%%%%%%%%%%%%%%%%%%%%%%%%%%%%%%%%%%%%%%%%%%%%%%%%%%%%%%%%%%%%%%%%%%%%%%%%
\subsection{Twin Norden metric of  Hom-Lie groups and Hom-Lie algebras}

In this section, we restrict ourselves to the case that $\langle\cdot,\cdot\rangle$ is a left-invariant Norden metric (or, a Norden metric) on a left-invariant almost complex Hom-Lie group $(G, \diamond, e_\Phi, \Phi)$ (or, an  almost complex Hom-Lie algebra  $(\mathfrak{g^!}, [\cdot, \cdot]_{\mathfrak{g^!}}, \phi_{\mathfrak{g^!}})$).  So we can apply $J$ to obtain a new left-invariant Norden metric (or, a Norden metric) $\ll\cdot,\cdot\gg$
associated with the left-invariant Norden metric (or, the Norden metric) $\langle\cdot,\cdot\rangle$ of the left-invariant almost Norden Hom-Lie group $(G, \diamond, e_\Phi, \Phi, J, \langle\cdot, \cdot\rangle)$ (or, the almost Norden Hom-Lie algebra $(\mathfrak{g^!}, [\cdot, \cdot]_{\mathfrak{g^!}}, \phi_\mathfrak{g^!}, J, \langle\cdot,\cdot\rangle)$)  defined by
\[
\ll x,y\gg=\langle ({ \phi_\mathfrak{g^!}}\circ J)x,y\rangle,\ \ \ \ \ 
\]
for any $x,y\in { \mathfrak{g^!}}$. In this case, the pure tensor  $\ll\cdot,\cdot\gg$ is called the twin metric of $\langle\cdot,\cdot\rangle$.
\begin{proposition}\label{L21}
	Let $(  J, \langle\cdot,\cdot\rangle)$ be a left-invariant almost Norden structure on a Hom-Lie group $(G, \diamond, e_\Phi, \Phi)$ (or, an almost  Norden  structure on a Hom-Lie algebra $(\mathfrak{g^!}, [\cdot, \cdot]_{\mathfrak{g^!}}, \phi_\mathfrak{g^!})$) and $\ll\cdot,\cdot\gg=\langle ({ \phi_\mathfrak{g^!}}\circ J)\cdot,\cdot\rangle$. Then the following identities hold
	\begin{align*}
&(i)\ ({\Phi}_{_{{ \phi_\mathfrak{g^!}}\circ J}}\ll\cdot,\cdot\gg)(x,y,z)=({\Phi}_{_{{ \phi_\mathfrak{g^!}}\circ J}}\langle\cdot,\cdot\rangle)(x, ({ \phi_\mathfrak{g^!}}\circ J)y,z)+\langle N_{ { \phi_\mathfrak{g^!}}\circ J}(x,y),{ \phi_\mathfrak{g^!}}(z)\rangle,\\
&(ii)\ (\nabla_x\ll\cdot, \cdot\gg)(y,z)=(\nabla_x\langle\cdot,\cdot \rangle)(({ \phi_\mathfrak{g^!}}\circ J)y,z)+\langle(	\nabla_x ({ \phi_\mathfrak{g^!}}\circ J))y,{ \phi_\mathfrak{g^!}}(z)\rangle,
	\end{align*}
	for any $x,y,z\in \mathfrak{g^!}$. 
\end{proposition}
\begin{proof}
	Using (\ref{L3}), we have
 	\begin{align*}
 	({\Phi}_{_{{ \phi_\mathfrak{g^!}}\circ J}}\ll\cdot,\cdot\gg)(x,y,z)=&\ll[y,({ \phi_\mathfrak{g^!}}\circ J)(x)]_{\mathfrak{g^!}}-({ \phi_\mathfrak{g^!}}\circ J)[y,x]_{\mathfrak{g^!}},
 	{ \phi_\mathfrak{g^!}}(z) \gg\\
 	&+
 	\ll{ \phi_\mathfrak{g^!}}(y) ,[z,({ \phi_\mathfrak{g^!}}\circ J)(x)]_{\mathfrak{g^!}}-({ \phi_\mathfrak{g^!}}\circ J)[z,x]_{\mathfrak{g^!}}
 	\gg.
 \end{align*}
 The above equation leads to
 	\begin{align*}
 ({\Phi}_{_{{ \phi_\mathfrak{g^!}}\circ J}}\ll\cdot,\cdot\gg)(x,y,z)=&\langle({ \phi_\mathfrak{g^!}}\circ J) [y,({ \phi_\mathfrak{g^!}}\circ J)(x)]_{\mathfrak{g^!}}-({ \phi_\mathfrak{g^!}}\circ J)[y,x]_{\mathfrak{g^!}},
 { \phi_\mathfrak{g^!}}(z) \rangle\\
 &+
 \langle ({ \phi_\mathfrak{g^!}}\circ J){ \phi_\mathfrak{g^!}}(y) ,[z,({ \phi_\mathfrak{g^!}}\circ J)(x)]_{\mathfrak{g^!}}-({ \phi_\mathfrak{g^!}}\circ J)[z,x]_{\mathfrak{g^!}}
 \rangle.
 \end{align*}
By adding and subtracting $ \langle({ \phi_\mathfrak{g^!}}\circ J) [x,({ \phi_\mathfrak{g^!}}\circ J)(y)]_{\mathfrak{g^!}}+[({ \phi_\mathfrak{g^!}}\circ J)x,({ \phi_\mathfrak{g^!}}\circ J)(y)]_{\mathfrak{g^!}}, { \phi_\mathfrak{g^!}}(z) \rangle$ in the above equation,
we conclude (i). From (\ref{L1}),  we have
\begin{align*}
(\nabla_x\ll\cdot, \cdot\gg)(y,z)=-\ll \nabla_x y,{ \phi_\mathfrak{g^!}}(z)\gg-\ll{ \phi_\mathfrak{g^!}}(y),\nabla_xz\gg,
\end{align*}
from which it follows 
\begin{align*}
(\nabla_x\ll\cdot,\cdot\gg)(y,z)=&-\langle ({ \phi_\mathfrak{g^!}}\circ J)(\nabla_x y),{ \phi_\mathfrak{g^!}}(z)\rangle-\langle({ \phi_\mathfrak{g^!}}\circ J)({ \phi_\mathfrak{g^!}}(y)),\nabla_x z\rangle\\
=&-\langle ({ \phi_\mathfrak{g^!}}\circ J)(\nabla_x y)+\nabla_x ({ \phi_\mathfrak{g^!}}\circ J)y-\nabla_x ({ \phi_\mathfrak{g^!}}\circ J)y,{ \phi_\mathfrak{g^!}}(z)\rangle-\langle({ \phi_\mathfrak{g^!}}\circ J)({ \phi_\mathfrak{g^!}}(y)),\nabla_x z\rangle.
\end{align*}
Applying (\ref{L1}) in the last equation, (ii) yields.
\end{proof}
%	\begin{corollary}
%	On the K\"{a}hler-Norden statistical manifold $(M, \nabla, g, J)$,  
%	\end{corollary}
As an immediate consequence of formula (i), we obtain that
in a left-invariant Norden  Hom-Lie group $(G, \diamond, e_\Phi, \Phi, J, \langle\cdot, \cdot\rangle)$ (or, a Norden Hom-Lie algebra $(\mathfrak{g^!}, [\cdot, \cdot]_{\mathfrak{g^!}}, \phi_{\mathfrak{g^!}},J,\langle\cdot,\cdot\rangle)$),  ${\Phi}_{_{{ \phi_\mathfrak{g^!}}\circ J}}\langle\cdot,\cdot\rangle=0$ if and only if ${\Phi}_{_{{ \phi_\mathfrak{g^!}}\circ J}}\ll\cdot,\cdot\gg=0$.
%%%%%%%%%%%%%%%%%%%%%%%%%%%%%%%%%%%%%%%%%%%%
\section{Left-invariant holomorphic Norden Hom-lie groups (holomorphic Norden Hom-Lie algebras)}
In this Section,
we give the definition of a left-invariant holomorphic Norden Hom-Lie group (or, a holomorphic Norden Hom-Lie algebra).
We show  that there exist a one-to-one correspondence between left invariant K\"{a}hler-Norden Hom-Lie groups (or, K\"{a}hler-Norden Hom-Lie algebras) and left-invariant holomorphic Norden Hom-Lie groups (or, holomorphic Norden Hom-Lie algebras).
 \begin{proposition}\label{MA88}
	Let $(G, \diamond, e_\Phi, \Phi, J, \langle\cdot, \cdot\rangle)$ be a left-invariant almost Norden Hom-Lie group (or, $( \mathfrak{g^!}, [\cdot, \cdot]_{\mathfrak{g^!}}, \phi_\mathfrak{g^!},J,\langle\cdot,\cdot\rangle)$ be an  almost Norden  Hom-Lie algebra). Then 
	\begin{align*}
	&(i)\ \ ({\Phi}_{_{{ \phi_\mathfrak{g^!}}\circ J}}\langle\cdot,\cdot\rangle)(x, y,z)=\langle (	\nabla_y( { \phi_\mathfrak{g^!}}\circ J))x,{ \phi_\mathfrak{g^!}}(z)\rangle+\langle { \phi_\mathfrak{g^!}}(y),  (	\nabla_z({ \phi_\mathfrak{g^!}}\circ J))x\rangle-
	\langle  (	\nabla_x({ \phi_\mathfrak{g^!}}\circ J))(y),{ \phi_\mathfrak{g^!}}z\rangle,\\
	&(ii)\ ({\Phi}_{_{{ \phi_\mathfrak{g^!}}\circ J}}\langle\cdot,\cdot\rangle)(x, y,z)+({\Phi}_{_{{ \phi_\mathfrak{g^!}}\circ J}}\langle\cdot,\cdot\rangle)( z,y,x)=2\langle (	\nabla_y ({ \phi_\mathfrak{g^!}}\circ J))x,{ \phi_\mathfrak{g^!}}(z)\rangle,
	\end{align*}
	for any $x,y,z\in \mathfrak{g^!}$,  where $\nabla$  is the left-invariant Hom-Levi-Civita connection on $G$ (or, the Hom-Levi-Civita connection on $\mathfrak{g^!}$). 
\end{proposition}
\begin{proof}
	From (\ref{L3}), we have
	\begin{align*}
	({\Phi}_{_{{ \phi_\mathfrak{g^!}}\circ J}}\langle\cdot,\cdot\rangle)(x, y,z)=
	\langle[y,({ \phi_\mathfrak{g^!}}\circ J)(x)]_{\mathfrak{g^!}}-({ \phi_\mathfrak{g^!}}\circ J)[y,x]_{\mathfrak{g^!}},
	{ \phi_\mathfrak{g^!}}(z) \rangle+
	\langle{ \phi_\mathfrak{g^!}}(y) ,[z,({ \phi_\mathfrak{g^!}}\circ J)(x)]_{\mathfrak{g^!}}-({ \phi_\mathfrak{g^!}}\circ J)[z,x]_{\mathfrak{g^!}}
	\rangle.
	\end{align*}
	Applying (i) of (\ref{twoo}) in the above equation yields
	\begin{align}
	({\Phi}_{_{{ \phi_\mathfrak{g^!}}\circ J}}\langle\cdot,\cdot\rangle)(x,y,z)=&
	\langle \nabla_y({ \phi_\mathfrak{g^!}}\circ J)(x)-\nabla_{({ \phi_\mathfrak{g^!}}\circ J)(x)}y-({ \phi_\mathfrak{g^!}}\circ J)(\nabla_y x-\nabla_x y),
	{ \phi_\mathfrak{g^!}}(z) \rangle\label{L20}\\
	&+
	\langle{ \phi_\mathfrak{g^!}}(y) ,\nabla_z({ \phi_\mathfrak{g^!}}\circ J)(x)-\nabla_{({ \phi_\mathfrak{g^!}}\circ J)(x)}z-({ \phi_\mathfrak{g^!}}\circ J)(\nabla_z x-\nabla_x z)
	\rangle.\nonumber
	\end{align}
	On the other hand, from (ii) of (\ref{twoo}) and Definition \ref{L100} it follows 
	\begin{align*}
	&
	\langle \nabla_{({ \phi_\mathfrak{g^!}}\circ J)(x)} y,
	{ \phi_\mathfrak{g^!}}(z) \rangle=-
	\langle{ \phi_\mathfrak{g^!}}(y),\nabla_{({ \phi_\mathfrak{g^!}}\circ J)(x)} z\rangle,\\
	&
	\langle{ \phi_\mathfrak{g^!}}(y) ,({ \phi_\mathfrak{g^!}}\circ J)(\nabla_x z)
	\rangle=-\langle \nabla_x({ \phi_\mathfrak{g^!}}\circ J) y,{ \phi_\mathfrak{g^!}}(z)
	\rangle.
	\end{align*}
	Setting two last equations in (\ref{L20}), we get (i).  Similarly like in the above, we have
	\begin{align*}
	&\ \ ({\Phi}_{_{{ \phi_\mathfrak{g^!}}\circ J}}\langle\cdot,\cdot\rangle)(z, y,x)=\langle (	\nabla_y ({ \phi_\mathfrak{g^!}}\circ J))z,{ \phi_\mathfrak{g^!}}(x)\rangle+\langle { \phi_\mathfrak{g^!}}(y),  (	\nabla_x({ \phi_\mathfrak{g^!}}\circ J))z-
	\langle  (	\nabla_z( { \phi_\mathfrak{g^!}}\circ J))y,{ \phi_\mathfrak{g^!}}(x)\rangle.
	\end{align*}
	The above equation and Theorem \ref{MA8} imply (ii).
\end{proof}
Assume $(J,\langle\cdot,\cdot\rangle)$ is a left-invariant almost Norden structure on a Hom-Lier group $(G, \diamond, e_\Phi, \Phi)$ (or, an almost Norden structure on a Hom-Lie algebra $( \mathfrak{g}, [\cdot, \cdot]_{\mathfrak{g^!}}, \phi_\mathfrak{g^!})$). We say that $(G, \diamond, e_\Phi, \Phi, J, \langle\cdot, \cdot\rangle)$ (or, $( \mathfrak{g^!}, [\cdot, \cdot]_{\mathfrak{g^!}}, \phi_\mathfrak{g^!},J,\langle\cdot,\cdot\rangle)$) is a left-invariant almost {\it holomorphic} Norden Hom-Lie group (or, an almost {\it holomorphic} Norden Hom-Lie algebra)
if
\begin{align*}
{\Phi}_{_{{ \phi_\mathfrak{g^!}}\circ J}}\langle\cdot,\cdot\rangle=0,
\end{align*}
i.e., 
	\begin{align}\label{L9}
\langle[y,({ \phi_\mathfrak{g^!}}\circ J)(x)]_{\mathfrak{g^!}},
{ \phi_\mathfrak{g^!}}(z) \rangle+
\langle{ \phi_\mathfrak{g^!}}(y) ,[z,({ \phi_\mathfrak{g^!}}\circ J)(x)]_{\mathfrak{g^!}}\rangle\!\!=\!\!\langle({ \phi_\mathfrak{g^!}}\circ J)[y,x]_{\mathfrak{g^!}},
{ \phi_\mathfrak{g^!}}(z) \rangle+\langle{ \phi_\mathfrak{g^!}}(y) ,({ \phi_\mathfrak{g^!}}\circ J)[z,x]_{\mathfrak{g^!}}
\rangle,
\end{align}
 for any $x,y,z\in\mathfrak{g^!}$.
Also, if $J$ is integrable, $(G, \diamond, e_\Phi, \Phi, J, \langle\cdot, \cdot\rangle)$ is called a left-invariant holomorphic Norden Hom-Lie group (or, $( \mathfrak{g^!}, [\cdot, \cdot]_{\mathfrak{g^!}}, \phi_\mathfrak{g^!},J,\langle\cdot,\cdot\rangle)$ is called a holomorphic Norden Hom-Lie algebra).
\begin{example}
	We consider the Norden Hom-Lie algebra  $(\mathfrak{g}, [\cdot, \cdot], \phi_{\mathfrak{g}}, J, \langle\cdot,\cdot\rangle)$ introduced in Example (\ref{L10}). Setting $x=y=w=e_1$ in (\ref{L9}), implies $A=0$. Considering $x=y=e_1$ and $w=e_2$, we get $B=0$. Also, putting  $x=y=e_1$ and $w=e_3$,  yields $D=0$. Moreover, from  $x=y=e_1$ and $w=e_4$,  we conclude $C=0$. It is a contradiction  with the condition non-degenerate of the  metric $\langle\cdot,\cdot\rangle$ and so $( \mathfrak{g}, [\cdot, \cdot], \phi_\mathfrak{g},J,\langle\cdot,\cdot\rangle)$ is not  a holomorphic Norden Hom-Lie algebra.
\end{example}
\begin{example}
Let  $(\mathfrak{g}, [\cdot, \cdot]_\mathfrak{g}, \phi_{\mathfrak{g}}, J, \langle,\rangle)$ be a $4$-dimensional   Norden Hom-Lie algebra in Example (\ref{L8}).  We result that all of
$
\langle[e_i,({ \phi_\mathfrak{g}}\circ J)(e_j)]_\mathfrak{g},
{ \phi_\mathfrak{g}}(e_k) \rangle+
\langle{ \phi_\mathfrak{g}}(e_i) ,[e_k,({ \phi_\mathfrak{g}}\circ J)(e_j)]_\mathfrak{g}\rangle=\langle({ \phi_\mathfrak{g}}\circ J)[e_i,e_j]_\mathfrak{g},
{ \phi_\mathfrak{g}}(e_k) \rangle+\langle{ \phi_\mathfrak{g}}(e_i) ,({ \phi_\mathfrak{g}}\circ J)[e_k,e_j]_\mathfrak{g}
\rangle$, $i,j,k=1,2,3,4,
$  are zero expect
\begin{align*}
\langle[e_1,({ \phi_\mathfrak{g}}\circ J)e_2]_\mathfrak{g},
{ \phi_\mathfrak{g}}(e_2) \rangle\!\!+\!\!
\langle{ \phi_\mathfrak{g}}(e_1) ,[e_2,({ \phi_\mathfrak{g}}\circ J)e_2]_\mathfrak{g}\rangle\!\!=&\!\!-\!\!B\!\!=\!\!\langle({ \phi_\mathfrak{g}}\circ J)[e_1,e_2]_\mathfrak{g},
{ \phi_\mathfrak{g}}(e_2) \rangle\!\!+\!\!\langle{ \phi_\mathfrak{g}}(e_1) ,({ \phi_\mathfrak{g}}\circ J)[e_2,e_2]_\mathfrak{g}
\rangle,\\
\langle[e_1,({ \phi_\mathfrak{g}}\circ J)e_3]_\mathfrak{g},
{ \phi_\mathfrak{g}}(e_3) \rangle\!\!+\!\!
\langle{ \phi_\mathfrak{g}}(e_1) ,[e_3,({ \phi_\mathfrak{g}}\circ J)e_3]_\mathfrak{g}\rangle\!\!=&B\!\!=\!\!\langle({ \phi_\mathfrak{g}}\circ J)[e_1,e_3]_\mathfrak{g},
{ \phi_\mathfrak{g}}(e_3) \rangle\!\!+\!\!\langle{ \phi_\mathfrak{g}}(e_1) ,({ \phi_\mathfrak{g}}\circ J)[e_3,e_3]_\mathfrak{g}
\rangle,\\
\langle[e_2,({ \phi_\mathfrak{g}}\circ J)e_1]_\mathfrak{g},
{ \phi_\mathfrak{g}}(e_2) \rangle\!\!+\!\!
\langle{ \phi_\mathfrak{g}}(e_2) ,[e_2,({ \phi_\mathfrak{g}}\circ J)e_1]_\mathfrak{g}\rangle\!\!=&2B\!\!=\!\!\langle({ \phi_\mathfrak{g}}\circ J)[e_2,e_1]_\mathfrak{g},
{ \phi_\mathfrak{g}}(e_2) \rangle\!\!+\!\!\langle{ \phi_\mathfrak{g}}(e_2) ,({ \phi_\mathfrak{g}}\circ J)[e_2,e_1]_\mathfrak{g}
\rangle,\\
\langle[e_2,({ \phi_\mathfrak{g}}\circ J)e_2]_\mathfrak{g},
{ \phi_\mathfrak{g}}(e_1) \rangle\!\!+\!\!
\langle{ \phi_\mathfrak{g}}(e_2) ,[e_1,({ \phi_\mathfrak{g}}\circ J)e_2]_\mathfrak{g}\rangle\!\!=&\!\!-B\!\!=\!\!\langle({ \phi_\mathfrak{g}}\circ J)[e_2,e_2]_\mathfrak{g},
{ \phi_\mathfrak{g}}(e_1) \rangle\!\!+\!\!\langle{ \phi_\mathfrak{g}}(e_2) ,({ \phi_\mathfrak{g}}\circ J)[e_1,e_2]_\mathfrak{g}
\rangle,\\
\langle[e_2,({ \phi_\mathfrak{g}}\circ J)e_3]_\mathfrak{g},
{ \phi_\mathfrak{g}}(e_4) \rangle\!\!+\!\!
\langle{ \phi_\mathfrak{g}}(e_2) ,[e_4,({ \phi_\mathfrak{g}}\circ J)e_3]_\mathfrak{g}\rangle\!\!=&\!\!-B\!\!=\!\!\langle({ \phi_\mathfrak{g}}\circ J)[e_2,e_3]_\mathfrak{g},
{ \phi_\mathfrak{g}}(e_4) \rangle\!\!+\!\!\langle{ \phi_\mathfrak{g}}(e_2) ,({ \phi_\mathfrak{g}}\circ J)[e_4,e_3]_\mathfrak{g}
\rangle,\\
\langle[e_2,({ \phi_\mathfrak{g}}\circ J)e_4]_\mathfrak{g},
{ \phi_\mathfrak{g}}(e_3) \rangle\!\!+\!\!
\langle{ \phi_\mathfrak{g}}(e_2) ,[e_3,({ \phi_\mathfrak{g}}\circ J)e_4]_\mathfrak{g}\rangle\!\!=&2B\!\!=\!\!\langle({ \phi_\mathfrak{g}}\circ J)[e_2,e_4]_\mathfrak{g},
{ \phi_\mathfrak{g}}(e_3) \rangle\!\!+\!\!\langle{ \phi_\mathfrak{g}}(e_2) ,({ \phi_\mathfrak{g}}\circ J)[e_3,e_4]_\mathfrak{g}
\rangle,\\
\langle[e_3,({ \phi_\mathfrak{g}}\circ J)e_1]_\mathfrak{g},
{ \phi_\mathfrak{g}}(e_3) \rangle\!\!+\!\!
\langle{ \phi_\mathfrak{g}}(e_3) ,[e_3,({ \phi_\mathfrak{g}}\circ J)e_1]_\mathfrak{g}\rangle\!\!=&\!\!-2B\!\!=\!\!\langle({ \phi_\mathfrak{g}}\circ J)[e_3,e_1]_\mathfrak{g},
{ \phi_\mathfrak{g}}(e_3) \rangle\!\!+\!\!\langle{ \phi_\mathfrak{g}}(e_3) ,({ \phi_\mathfrak{g}}\circ J)[e_3,e_1]_\mathfrak{g}
\rangle,\\
\langle[e_3,({ \phi_\mathfrak{g}}\circ J)e_2]_\mathfrak{g},
{ \phi_\mathfrak{g}}(e_4) \rangle\!\!+\!\!
\langle{ \phi_\mathfrak{g}}(e_3) ,[e_4,({ \phi_\mathfrak{g}}\circ J)e_2]_\mathfrak{g}\rangle\!\!=&\!\!-B\!\!=\!\!\langle({ \phi_\mathfrak{g}}\circ J)[e_3,e_2]_\mathfrak{g},
{ \phi_\mathfrak{g}}(e_4) \rangle\!\!+\!\!\langle{ \phi_\mathfrak{g}}(e_3) ,({ \phi_\mathfrak{g}}\circ J)[e_4,e_2]_\mathfrak{g}
\rangle,\\
\langle[e_3,({ \phi_\mathfrak{g}}\circ J)e_3]_\mathfrak{g},
{ \phi_\mathfrak{g}}(e_1) \rangle\!\!+\!\!
\langle{ \phi_\mathfrak{g}}(e_3) ,[e_1,({ \phi_\mathfrak{g}}\circ J)e_3]_\mathfrak{g}\rangle\!\!=&B\!\!=\!\!\langle({ \phi_\mathfrak{g}}\circ J)[e_3,e_3]_\mathfrak{g},
{ \phi_\mathfrak{g}}(e_1) \rangle\!\!+\!\!\langle{ \phi_\mathfrak{g}}(e_3) ,({ \phi_\mathfrak{g}}\circ J)[e_1,e_3]_\mathfrak{g}
\rangle,\\
\langle[e_3,({ \phi_\mathfrak{g}}\circ J)e_4]_\mathfrak{g},
{ \phi_\mathfrak{g}}(e_2) \rangle\!\!+\!\!
\langle{ \phi_\mathfrak{g}}(e_3) ,[e_2,({ \phi_\mathfrak{g}}\circ J)e_4]_\mathfrak{g}\rangle\!\!=&2B\!\!=\!\!\langle({ \phi_\mathfrak{g}}\circ J)[e_3,e_4]_\mathfrak{g},
{ \phi_\mathfrak{g}}(e_2) \rangle\!\!+\!\!\langle{ \phi_\mathfrak{g}}(e_3) ,({ \phi_\mathfrak{g}}\circ J)[e_2,e_4]_\mathfrak{g}
\rangle,\\
\langle[e_4,({ \phi_\mathfrak{g}}\circ J)e_2]_\mathfrak{g},
{ \phi_\mathfrak{g}}(e_3) \rangle\!\!+\!\!
\langle{ \phi_\mathfrak{g}}(e_4) ,[e_3,({ \phi_\mathfrak{g}}\circ J)e_2]_\mathfrak{g}\rangle\!\!=&\!\!-B\!\!=\!\!\langle({ \phi_\mathfrak{g}}\circ J)[e_4,e_2]_\mathfrak{g},
{ \phi_\mathfrak{g}}(e_3) \rangle\!\!+\!\!\langle{ \phi_\mathfrak{g}}(e_4) ,({ \phi_\mathfrak{g}}\circ J)[e_3,e_2]_\mathfrak{g}
\rangle,\\
\langle[e_4,({ \phi_\mathfrak{g}}\circ J)e_3]_\mathfrak{g},
{ \phi_\mathfrak{g}}(e_2) \rangle\!\!+\!\!
\langle{ \phi_\mathfrak{g}}(e_4) ,[e_2,({ \phi_\mathfrak{g}}\circ J)e_3]_\mathfrak{g}\rangle\!\!=&\!\!-B\!\!=\!\!\langle({ \phi_\mathfrak{g}}\circ J)[e_4,e_3]_\mathfrak{g},
{ \phi_\mathfrak{g}}(e_2) \rangle\!\!+\!\!\langle{ \phi_\mathfrak{g}}(e_4) ,({ \phi_\mathfrak{g}}\circ J)[e_2,e_3]_\mathfrak{g}
\rangle,
\end{align*}
which imply that, (\ref{L9}) holds. Thus  $( \mathfrak{g}, [\cdot, \cdot]_\mathfrak{g}, \phi_\mathfrak{g},J,\langle\cdot,\cdot\rangle)$ is  a holomorphic Norden Hom-Lie algebra.
\end{example}
\begin{corollary}\label{L5}
	A left-invariant almost Norden Hom-Lie group $(G, \diamond, e_\Phi, \Phi, J, \langle\cdot, \cdot\rangle)$ (or, an almost Norden Hom-Lie algebra $( \mathfrak{g^!}, [\cdot, \cdot]_{\mathfrak{g^!}}, \phi_\mathfrak{g^!},J,\langle\cdot,\cdot\rangle)$) is holomorphic if and only if the  structure ${ \phi_\mathfrak{g^!}}\circ J$ is invariant with respect to the left-invariant Hom-Levi-Civita connection (or, the Hom-Levi-Civita connection) $\nabla$, i.e., 
	\[
	\nabla({ \phi_\mathfrak{g^!}}\circ J)=0.
	\]
%	for any $x\in \mathfrak{g^!}$.
\end{corollary}
Applying  Corollary \ref{L5}, we can deduce
\begin{proposition}\label{MA17}
	Let $(G, \diamond, e_\Phi, \Phi, J)$ be a left-invariant almost complex Hom-Lie group (or, $(\mathfrak{g^!}, [\cdot, \cdot]_{\mathfrak{g^!}}, \phi_\mathfrak{g^!}, J)$ be
	an almost complex Hom-Lie algebra)  equipped with a left-invariant pseudo-Riemannian metric (or, a pseudo-Riemannian metric)	$\langle\cdot,\cdot\rangle$ and a left-invariant connection (or, connection) $\nabla$.
	Then the following statements are equivalent:\\
	a)\ $(G, \diamond, e_\Phi, \Phi, J, \langle\cdot, \cdot\rangle)$ is a left-invariant K\"{a}hler-Norden Hom-Lie group (or, $(\mathfrak{g^!}, [\cdot, \cdot]_{\mathfrak{g^!}}, \phi_{\mathfrak{g^!}},J,\langle\cdot,\cdot\rangle)$ is a K\"{a}hler-Norden Hom-Lie algebra),\\
	b)\ $(G, \diamond, e_\Phi, \Phi, J, \langle\cdot, \cdot\rangle)$ is a left-invariant holomorphic Norden Hom-Lie group (or, $(\mathfrak{g^!}, [\cdot, \cdot]_{\mathfrak{g^!}}, \phi_{\mathfrak{g^!}},J,\langle\cdot,\cdot\rangle)$ is 
a holomorphic Norden Hom-Lie algebra).
\end{proposition}
	\begin{corollary}
	On a left-invariant K\"{a}hler-Norden Hom-Lie group $(G, \diamond, e_\Phi, \Phi, J, \langle\cdot, \cdot\rangle)$ (or, a K\"{a}hler-Norden Hom-Lie algebra $(\mathfrak{g^!}, [\cdot, \cdot]_{\mathfrak{g^!}}, \phi_{\mathfrak{g^!}},J,\langle\cdot,\cdot\rangle)$), we have
		\begin{align*}
	{\Phi}_{_{{ \phi_\mathfrak{g^!}}\circ J}}\langle\cdot,\cdot\rangle=0=
	{\Phi}_{_{{ \phi_\mathfrak{g^!}}\circ J}}\ll\cdot,\cdot\gg.
	\end{align*}
		\end{corollary}
	\begin{corollary}
		Let $(G, \diamond, e_\Phi, \Phi, J, \langle\cdot, \cdot\rangle)$ be a left-invariant K\"{a}hler-Norden Hom-Lie group (or, $(\mathfrak{g^!}, [\cdot, \cdot]_{\mathfrak{g^!}}, \phi_{\mathfrak{g^!}},J,\langle\cdot,\cdot\rangle)$ be  a K\"{a}hler-Norden Hom-Lie algebra) and $\ll \cdot,\cdot\gg=\langle ({ \phi_\mathfrak{g^!}}\circ J)\cdot,\cdot\rangle$. Then the following statements hold:\\
		a)\  $(G, \diamond, e_\Phi, \Phi, J, \ll\cdot, \cdot\gg)$ is a left-invariant K\"{a}hler-Norden Hom-Lie group (or, $(\mathfrak{g^!}, [\cdot, \cdot]_{\mathfrak{g^!}}, \phi_{\mathfrak{g^!}},J,\ll\cdot,\cdot\gg)$ is a K\"{a}hler-Norden Hom-Lie algebra),\\
			b)\  $(G, \diamond, e_\Phi, \Phi, J, \ll\cdot, \cdot\gg)$ is a left-invariant holomorphic Norden Hom-Lie group (or, $(\mathfrak{g^!}, [\cdot, \cdot]_{\mathfrak{g^!}}, \phi_{\mathfrak{g^!}},J,\ll\cdot,\cdot\gg)$ is 
	a	holomorphic Norden Hom-Lie algebra).
	\end{corollary}
\section{ curvature tensors in  holomorphic Norden Hom-Lie algebras}
We present some properties of  the Riemannian curvature tensor  of
a left-invariant holomorphic Norden Hom-Lie group (or, a holomorphic Norden Hom-Lie algebra) using the generalization of Tachibana operators, in this Section. Also, we show that any left-invariant holomorphic Hom-Lie group is a flat (or, holomorphic Norden Hom-Lie algebra carries  a Hom-Left-symmetric algebra )  if its left-invariant complex structure (or, complex structure) is abelian. 

 Let $(G, \diamond, e_\Phi, \Phi, J, \langle\cdot, \cdot\rangle)$ be a left-invariant K\"{a}hler-Norden Hom-Lie group (or, $(\mathfrak{g^!}, [\cdot, \cdot]_{\mathfrak{g^!}}, \phi_{\mathfrak{g^!}},J,\langle\cdot,\cdot\rangle)$ be a K\"{a}hler-Norden Hom-Lie algebra). The curvature tensor $\mathcal{K}$ of $G$ (or, $\mathfrak{g^!}$) is defined by
\begin{equation}\label{SS8}
\mathcal{K}(x,y):=\nabla_{\phi_{\mathfrak{g^!}}(x)}\circ \nabla_y-\nabla_{ \phi_{\mathfrak{g^!}}(y)}\circ \nabla_x-\nabla_{[x,y]_{\mathfrak{g^!}}}\circ \phi_{\mathfrak{g^!}},
\end{equation}
for any $x, y\in \mathfrak{g^!}$. Also, left-invariant K\"{a}hler-Norden Hom-Lie group  is said to be flat if  its curvature tensor vanishes identically.
	It is easy to see that the following formulas hold
\begin{align}
&\langle\mathcal{K}(x,y)z,w\rangle=-\langle\mathcal{K}(y,x)z,w\rangle,\nonumber\\
&\langle\mathcal{K}(x,y)z,w\rangle=-\langle\mathcal{K}(x,y)w,z\rangle,\nonumber\\
&\mathcal{K}(x,y)z+\mathcal{K}(y,z)x+\mathcal{K}(z,x)y=0,\nonumber\\
&(\nabla_x\mathcal{K})(y,z,w,t)+(\nabla_y\mathcal{K})(z,x,w,t)+(\nabla_z\mathcal{K})(x,y,w,t)=0.\label{L11}
\end{align}
We set
\[
\mathcal{K}(x,y,z,w)=\langle\mathcal{K}(x,y)z,w\rangle,
\]
for any $x,y,z,w\in\mathfrak{g^!}$. The condition $	\nabla({ \phi_\mathfrak{g^!}}\circ J)=0$, leads to
\[
\mathcal{K}(x,y)({ \phi_\mathfrak{g^!}}\circ J)z =({ \phi_\mathfrak{g^!}}\circ J)\mathcal{K}(x,y)z.
\]
Hence,  we have 
\begin{align*}
&\mathcal{K}(x,y,({ \phi_\mathfrak{g^!}}\circ J)z,w)=\mathcal{K}(x,y,z,({ \phi_\mathfrak{g^!}}\circ J)w),\\
&\mathcal{K}(({ \phi_\mathfrak{g^!}}\circ J)x,y,z,w)=\mathcal{K}(x,({ \phi_\mathfrak{g^!}}\circ J)y,z,w),
\end{align*}
i.e., $\mathcal{K}$ is pure with respect to $z$ and $w$, and also pure  with respect to $x$ and $y$.
Consider $\ll\cdot,\cdot\gg=\langle ({ \phi_\mathfrak{g^!}}\circ J)\cdot,\cdot\rangle$ and $\widetilde{\mathcal{K}}$ as the curvature tensor of $\ll\cdot,\cdot\gg$. By means of the Proposition \ref{L21}, $\mathcal{K}=\widetilde{\mathcal{K}}$  and considering $\widetilde{\mathcal{K}}(x,y,z,w)=\ll \widetilde{\mathcal{K}}(x,y)z,w\gg$, we get
\begin{align*}
&\widetilde{\mathcal{K}}(x,y,z,w)=\ll \widetilde{\mathcal{K}}(x,y)z,w\gg=\langle({ \phi_\mathfrak{g^!}}\circ J)\widetilde{\mathcal{K}}(x,y)z,w\rangle\\
&=\langle\widetilde{\mathcal{K}}(x,y)z,({ \phi_\mathfrak{g^!}}\circ J)w\rangle=\langle{\mathcal{K}}(x,y)z,({ \phi_\mathfrak{g^!}}\circ J)w\rangle={\mathcal{K}}(x,y,z,({ \phi_\mathfrak{g^!}}\circ J)w),
\end{align*}
and 
\begin{align*}
&\widetilde{\mathcal{K}}(z,w,x,y)=\ll \widetilde{\mathcal{K}}(z,w)x,y\gg=\langle({ \phi_\mathfrak{g^!}}\circ J)\widetilde{\mathcal{K}}(z,w)x,y\rangle\\
&=\langle\widetilde{\mathcal{K}}(z,w)x,({ \phi_\mathfrak{g^!}}\circ J)y\rangle=\langle{\mathcal{K}}(z,w)x,({ \phi_\mathfrak{g^!}}\circ J)y\rangle={\mathcal{K}}(z,w,x,({ \phi_\mathfrak{g^!}}\circ J)y).
\end{align*}
Note that $\widetilde{\mathcal{K}}(x,y,z,w)=\widetilde{\mathcal{K}}(z,w,x,y)$,  thus the above equations imply
\[
{\mathcal{K}}(x,y,z,({ \phi_\mathfrak{g^!}}\circ J)w)={\mathcal{K}}(z,w,x,({ \phi_\mathfrak{g^!}}\circ J)y),
\]
which means that
${\mathcal{K}}(x,y,z,w)$ is pure with respect to $y$ and $w$. Thus  ${\mathcal{K}}(x,y,z,w)$ is pure.
\begin{proposition}
	The curvature tensor  of a left-invariant holomorphic Norden  Hom-Lie group (or, a holomorphic Norden Hom-Lie algebra) is pure.
\end{proposition}
\begin{theorem}
	In  a left-invariant holomorphic Norden Hom-Lie group $(G, \diamond, e_\Phi, \Phi, J, \langle\cdot, \cdot\rangle)$ (or, a holomorphic Norden Hom-Lie algebra $(\mathfrak{g^!}, [\cdot, \cdot]_{\mathfrak{g^!}}, \phi_{\mathfrak{g^!}},J,\langle\cdot,\cdot\rangle)$), 	the curvature tensor  is a  holomorphic tensor, i.e., 
	\[
	{\Phi}_{_{{ \phi_\mathfrak{g^!}}\circ J}}\mathcal{K}=0.
	\]
\end{theorem}
\begin{proof}
	According to (\ref{L3}), we have 
	\begin{align*}
	({\Phi}_{_{{ \phi_\mathfrak{g^!}}\circ J}}\mathcal{K})(x,y,z,w,t)=&\ \mathcal{K}(
	[y,({ \phi_\mathfrak{g^!}}\circ J)(x)]_{\mathfrak{g^!}}-({ \phi_\mathfrak{g^!}}\circ J)[y,x]_{\mathfrak{g^!}},
	{ \phi_\mathfrak{g^!}}(z),	{ \phi_\mathfrak{g^!}}(w), 	{ \phi_\mathfrak{g^!}}(t))\\
	&+\mathcal{K}({ \phi_\mathfrak{g^!}}(y),
	[z,({ \phi_\mathfrak{g^!}}\circ J)(x)]_{\mathfrak{g^!}}-({ \phi_\mathfrak{g^!}}\circ J)[z,x]_{\mathfrak{g^!}},
		{ \phi_\mathfrak{g^!}}(w), 	{ \phi_\mathfrak{g^!}}(t))\\
		&+\mathcal{K}({ \phi_\mathfrak{g^!}}(y),{ \phi_\mathfrak{g^!}}(z), 
		[w,({ \phi_\mathfrak{g^!}}\circ J)(x)]_{\mathfrak{g^!}}-({ \phi_\mathfrak{g^!}}\circ J)[w,x]_{\mathfrak{g^!}},
			{ \phi_\mathfrak{g^!}}(t))\\
			&+\mathcal{K}({ \phi_\mathfrak{g^!}}(y),{ \phi_\mathfrak{g^!}}(z), { \phi_\mathfrak{g^!}}(w),
			[t,({ \phi_\mathfrak{g^!}}\circ J)(x)]_{\mathfrak{g^!}}-({ \phi_\mathfrak{g^!}}\circ J)[t,x]_{\mathfrak{g^!}}
			).
	\end{align*}
	From the condition $\nabla( { \phi_\mathfrak{g^!}}\circ J)=({ \phi_\mathfrak{g^!}}\circ J)\nabla$ and  the last equation imply
	\begin{align*}
({\Phi}_{_{{ \phi_\mathfrak{g^!}}\circ J}}\mathcal{K})(x,y,z,w,t)=&\ \mathcal{K}(
-\nabla_{({ \phi_\mathfrak{g^!}}\circ J)(x)} y+({ \phi_\mathfrak{g^!}}\circ J)(\nabla_x y),
{ \phi_\mathfrak{g^!}}(z),	{ \phi_\mathfrak{g^!}}(w), 	{ \phi_\mathfrak{g^!}}(t))\\
&+\mathcal{K}({ \phi_\mathfrak{g^!}}(y),
-\nabla_{({ \phi_\mathfrak{g^!}}\circ J)(x)} z+({ \phi_\mathfrak{g^!}}\circ J)(\nabla_x z),
{ \phi_\mathfrak{g^!}}(w), 	{ \phi_\mathfrak{g^!}}(t))\\
&+\mathcal{K}({ \phi_\mathfrak{g^!}}(y),{ \phi_\mathfrak{g^!}}(z), 
-\nabla_{({ \phi_\mathfrak{g^!}}\circ J)(x)}w+({ \phi_\mathfrak{g^!}}\circ J)(\nabla_x w),
{ \phi_\mathfrak{g^!}}(t))\\
&+\mathcal{K}({ \phi_\mathfrak{g^!}}(y),{ \phi_\mathfrak{g^!}}(z), { \phi_\mathfrak{g^!}}(w),
-\nabla_{({ \phi_\mathfrak{g^!}}\circ J)(x)} t+({ \phi_\mathfrak{g^!}}\circ J)(\nabla_x t)
).
\end{align*}
	 Using again   (\ref{L3}) in the above equation, it follows
	\begin{align}\label{L12}
	({\Phi}_{_{{ \phi_\mathfrak{g^!}}\circ J}}\mathcal{K})(x,y,z,w,t)=(\nabla_{ (\phi_\mathfrak{g^!}\circ J)x}\mathcal{K})(y,z,w,t)-(\nabla_{x}\mathcal{K})( (\phi_\mathfrak{g^!}\circ J)y,z,w,t).
	\end{align}
	On the other hand, from (\ref{L11}) we have
	\begin{align*}
	(\nabla_{ (\phi_\mathfrak{g^!}\circ J)x}\mathcal{K})(y,z,w,t)=-(\nabla_{ y}\mathcal{K})(z,(\phi_\mathfrak{g^!}\circ J)x,w,t)-(\nabla_{z}\mathcal{K})((\phi_\mathfrak{g^!}\circ J)x,y,w,t).
	\end{align*}
	Since ${\mathcal{K}}$ is pure, thus from the above equation yields
		\begin{align*}
	(\nabla_{ (\phi_\mathfrak{g^!}\circ J)x}\mathcal{K})(y,z,w,t)=-(\nabla_{ y}\mathcal{K})(z,x,w,(\phi_\mathfrak{g^!}\circ J)t)-(\nabla_{z}\mathcal{K})(x,y,w,(\phi_\mathfrak{g^!}\circ J)t),
	\end{align*}
and also, we get
	\[
	(\nabla_{x}\mathcal{K})( (\phi_\mathfrak{g^!}\circ J)y,z,w,t)=(\nabla_{x}\mathcal{K})( y,z,w,(\phi_\mathfrak{g^!}\circ J)t).
	\]
	Using (\ref{L11}) and two last equations in (\ref{L12}), we conclude the assertion.
\end{proof}
\begin{proposition}\label{L31}
	Let $(G, \diamond, e_\Phi, \Phi, J, \langle\cdot, \cdot\rangle)$ be a left-invariant holomorphic Norden Hom-Lie group (or, $(\mathfrak{g^!}, [\cdot, \cdot]_{\mathfrak{g^!}}, \phi_{\mathfrak{g^!}},J,\langle,\rangle)$ be 
	 a holomorphic Norden Hom-Lie algebra) such that $J$ is a left-invariant abelin complex structure (or, an abelian complex structure). Then we have
\[
\nabla_{\phi_{\mathfrak{g^!}}(x)}\circ \nabla_y=\nabla_{ \phi_{\mathfrak{g^!}}(y)}\circ \nabla_x,
\]
for any $x,y\in \mathfrak{g^!}$.
\end{proposition}
\begin{proof}
	Using Proposition \ref{L30}, we can write
	\begin{align*}
(\nabla_{\phi_{\mathfrak{g^!}}(x)} \nabla_y)z=&\frac{1}{2}\nabla_{\phi_{\mathfrak{g^!}}(x)}([y,z]_{\mathfrak{g^!}}-({ \phi_\mathfrak{g}}\circ J)[y,({ \phi_\mathfrak{g^!}}\circ J)z]_{\mathfrak{g^!}})=\frac{1}{4}([{\phi_{\mathfrak{g^!}}(x)},[y,z]_{\mathfrak{g^!}}-({ \phi_\mathfrak{g}}\circ J)[y,({ \phi_\mathfrak{g^!}}\circ J)z]_{\mathfrak{g^!}}]_{\mathfrak{g^!}}\\
&-({ \phi_\mathfrak{g^!}}\circ J)[{\phi_{\mathfrak{g^!}}(x)},
({ \phi_\mathfrak{g^!}}\circ J)[y,z]_{\mathfrak{g^!}}]_{\mathfrak{g^!}}-({ \phi_\mathfrak{g^!}}\circ J)[{\phi_{\mathfrak{g^!}}(x)},[y,({ \phi_\mathfrak{g^!}}\circ J)z]_{\mathfrak{g^!}}
]_{\mathfrak{g^!}}).
\end{align*}
	Applying (\ref{L22}) in the above equation yields
		\begin{align*}
	(\nabla_{\phi_{\mathfrak{g^!}}(x)} \nabla_y)z=&\frac{1}{4}([{\phi_{\mathfrak{g^!}}(x)},[y,z]_{\mathfrak{g^!}}]_{\mathfrak{g^!}}
	-[({ \phi_\mathfrak{g^!}}\circ J)({\phi_{\mathfrak{g^!}}(x))},[({ \phi_\mathfrak{g^!}}\circ J)y,z]_{\mathfrak{g^!}}]_{\mathfrak{g^!}}\\
	&+({ \phi_\mathfrak{g^!}}\circ J)[	({ \phi_\mathfrak{g^!}}\circ J)){\phi_{\mathfrak{g^!}}(x)}),
	[({ \phi_\mathfrak{g^!}}\circ J)y,({ \phi_\mathfrak{g^!}}\circ J)z]_{\mathfrak{g^!}}]_{\mathfrak{g^!}}-({ \phi_\mathfrak{g^!}}\circ J)[{\phi_{\mathfrak{g^!}}(x)},[y,({ \phi_\mathfrak{g^!}}\circ J)z]_{\mathfrak{g^!}}
	]_{\mathfrak{g^!}}).
	\end{align*}
	Similarly, we get
		\begin{align*}
	(\nabla_{\phi_{\mathfrak{g^!}}(y)} \nabla_x)z=&\frac{1}{4}([{\phi_{\mathfrak{g^!}}(y)},[x,z]_{\mathfrak{g^!}}]_{\mathfrak{g^!}}
	-[({ \phi_\mathfrak{g^!}}\circ J)({\phi_{\mathfrak{g^!}}(y))},[({ \phi_\mathfrak{g^!}}\circ J)x,z]_{\mathfrak{g^!}}]_{\mathfrak{g^!}}\\
	&+({ \phi_\mathfrak{g^!}}\circ J)[	({ \phi_\mathfrak{g^!}}\circ J)){\phi_{\mathfrak{g^!}}(y)}),
	[({ \phi_\mathfrak{g^!}}\circ J)x,({ \phi_\mathfrak{g^!}}\circ J)z]_{\mathfrak{g^!}}]_{\mathfrak{g^!}}-({ \phi_\mathfrak{g^!}}\circ J)[{\phi_{\mathfrak{g^!}}(y)},[x,({ \phi_\mathfrak{g^!}}\circ J)z]_{\mathfrak{g^!}}
	]_{\mathfrak{g^!}}).
	\end{align*}
	Subtracting two above equations and using the Hom-Jacobi identity, we obtain
		\begin{align*}
	(\nabla_{\phi_{\mathfrak{g^!}}(x)}\nabla_y&-\nabla_{\phi_{\mathfrak{g^!}}(y)} \nabla_x)z=\frac{1}{4}(-[{\phi_{\mathfrak{g^!}}(z)},[x,y]_{\mathfrak{g^!}}]_{\mathfrak{g^!}}
	+[{\phi_{\mathfrak{g^!}}(z)},[({ \phi_\mathfrak{g^!}}\circ J)x,({ \phi_\mathfrak{g^!}}\circ J)y]_{\mathfrak{g^!}}]_{\mathfrak{g^!}}\\
	&-({ \phi_\mathfrak{g^!}}\circ J)[	({ \phi_\mathfrak{g^!}}\circ J)){\phi_{\mathfrak{g^!}}(z)}),
	[({ \phi_\mathfrak{g^!}}\circ J)x,({ \phi_\mathfrak{g^!}}\circ J)y]_{\mathfrak{g^!}}]_{\mathfrak{g^!}}+({ \phi_\mathfrak{g^!}}\circ J)[{\phi_{\mathfrak{g^!}}(z)},[x,({ \phi_\mathfrak{g^!}}\circ J)y]_{\mathfrak{g^!}}
	]_{\mathfrak{g^!}}).
	\end{align*}
	From (\ref{L22}) and the last equation, we deduce the assertion. 
	\end{proof}
A Hom-Left-symmetric algebra is a Hom-algebra $(V, \cdot,  \phi_V )$ such that
\begin{equation}\label{L95}
ass_{\phi_V}(u,v,w)=ass_{\phi_V}(v,u,w),
\end{equation}
where
\begin{align*}
ass_{\phi_V}(u,v,w)=(u\cdot v)\cdot\phi_V(w)-\phi_V(u)\cdot(v\cdot w),
\end{align*}
for any $u,v,w\in V$.  This relation is equivalent to the vanishing of the curvature of $(V, \cdot,  \phi_V )$ with the {\it commutator} 
on $V$ is given by $[u,v]_V=u\cdot v-v\cdot u$.

	A Hom-algebra $(V, \cdot, \phi_V)$ is called
	Hom-Lie-admissible algebra
	if its commutator bracket
	satisfies the Hom-Jacobi identity.
	For a Hom-Lie-admissible algebra we have $\circlearrowleft_{u,v,w} \mathcal{K}(u,v)w=0$.
	If $(V,\cdot, \phi_V)$ is a Hom-Lie-admissible algebra, then $(V, [\cdot ,\cdot ]_V, \phi_V)$ is a Hom-Lie algebra, where $[\cdot ,\cdot ]_V$ is the commutator bracket.
\begin{proposition}\label{3.10}
	A Hom-Left-symmetric algebra is a Hom-Lie-admissible algebra.
\end{proposition}
\begin{theorem}
Any left-invariant holomorphic Norden Hom-Lie group $(G, \diamond, e_\Phi, \Phi, J, \langle\cdot, \cdot\rangle)$ is a flat Hom-Lie group (or,  holomorphic Norden Hom-Lie algebra 	$(\mathfrak{g^!}, [\cdot, \cdot]_{\mathfrak{g^!}}, \phi_{\mathfrak{g^!}},J,\langle\cdot,\cdot\rangle)$ carries  a Hom-Left-symmetric algebra) if $J$ is an abelian complex structure. 
\end{theorem}
\begin{proof}
	To prove, sufficient we show that $\mathcal{K}(x,y)z=0$, for any $x,y,z\in\mathfrak{g^!}$. Using (\ref{L19}) and (\ref{L18}), we conclude that 
	\begin{align}
\mathcal{K}(({ \phi_\mathfrak{g^!}}\circ J)x,({ \phi_\mathfrak{g^!}}\circ J)y)z=-\mathcal{K}(x,y)z.
	\end{align}
	On the other hand, Proposition \ref{L31} implies 
	\begin{align*}
\mathcal{K}(({ \phi_\mathfrak{g^!}}\circ J)x,({ \phi_\mathfrak{g^!}}\circ J)y)z=-\nabla_{[({ \phi_\mathfrak{g^!}}\circ J)x,({ \phi_\mathfrak{g^!}}\circ J)y]_{\mathfrak{g^!}}} \phi_{\mathfrak{g^!}}=-\nabla_{[x,y]_{\mathfrak{g^!}}} \phi_{\mathfrak{g^!}}=\mathcal{K}(x,y)z.
	\end{align*}
	From the above equations we conclude the assertion.
	\end{proof}
\begin{example}
	We consider a holomorphic Norden Hom-Lie algebra 	$(\mathfrak{g}, [\cdot, \cdot]_\mathfrak{g}, \phi_{\mathfrak{g}},J,\langle \cdot,\cdot\rangle)$ in Example (\ref{L94}). For 
the Hom-Levi-Civita connection $\nabla$, we get 
\begin{align*}
\nabla_{\nabla_{e_i} e_j}\phi_{\mathfrak{g}}(e_k)-\nabla_{\phi_{\mathfrak{g}}(e_i)}\nabla_{e_j} e_k=0=\nabla_{\nabla_{e_j} e_i}\phi_{\mathfrak{g}}(e_k)-\nabla_{\phi_{\mathfrak{g}}(e_j)}\nabla_{e_i} e_k, \ \ \ \forall i,j,k=1,\cdots,6,
\end{align*}
 expect 
 \begin{align*}
 &\nabla_{\nabla_{e_5} e_5}\phi_{\mathfrak{g}}(e_5)-\nabla_{\phi_{\mathfrak{g}}(e_5)}\nabla_{e_5} e_5=-\frac{1}{2}e_1=\nabla_{\nabla_{e_5} e_5}\phi_{\mathfrak{g}}(e_5)-\nabla_{\phi_{\mathfrak{g}}(e_5)}\nabla_{e_5} e_5, \\
 &\nabla_{\nabla_{e_5} e_5}\phi_{\mathfrak{g}}(e_6)-\nabla_{\phi_{\mathfrak{g}}(e_5)}\nabla_{e_5} e_6=\frac{1}{2}e_2=\nabla_{\nabla_{e_5} e_5}\phi_{\mathfrak{g}}(e_6)-\nabla_{\phi_{\mathfrak{g}}(e_5)}\nabla_{e_5} e_6, \\
 &\nabla_{\nabla_{e_5} e_6}\phi_{\mathfrak{g}}(e_5)-\nabla_{\phi_{\mathfrak{g}}(e_5)}\nabla_{e_6} e_5=-\frac{1}{2}e_2=\nabla_{\nabla_{e_6} e_5}\phi_{\mathfrak{g}}(e_5)-\nabla_{\phi_{\mathfrak{g}}(e_6)}\nabla_{e_5} e_5,\\
 &\nabla_{\nabla_{e_5} e_6}\phi_{\mathfrak{g}}(e_6)-\nabla_{\phi_{\mathfrak{g}}(e_5)}\nabla_{e_6} e_6=-\frac{1}{2}e_1=\nabla_{\nabla_{e_6} e_5}\phi_{\mathfrak{g}}(e_6)-\nabla_{\phi_{\mathfrak{g}}(e_6)}\nabla_{e_5} e_6,\\
 &\nabla_{\nabla_{e_6} e_6}\phi_{\mathfrak{g}}(e_5)-\nabla_{\phi_{\mathfrak{g}}(e_6)}\nabla_{e_6} e_5=\frac{1}{2}e_1=\nabla_{\nabla_{e_6} e_6}\phi_{\mathfrak{g}}(e_5)-\nabla_{\phi_{\mathfrak{g}}(e_6)}\nabla_{e_6} e_5,\\
  &\nabla_{\nabla_{e_6} e_6}\phi_{\mathfrak{g}}(e_6)-\nabla_{\phi_{\mathfrak{g}}(e_6)}\nabla_{e_6} e_6=-\frac{1}{2}e_2=\nabla_{\nabla_{e_6} e_6}\phi_{\mathfrak{g}}(e_6)-\nabla_{\phi_{\mathfrak{g}}(e_6)}\nabla_{e_6} e_6.
 %  &(e_5\cdot e_5)\cdot\phi_{\mathfrak{g}}(e_5)-\phi_{\mathfrak{g}}(e_5)\cdot(e_5\cdot e_5)=\frac{1}{2}e_2=(e_5\cdot e_5)\cdot\phi_{\mathfrak{g}}(e_5)-\phi_{\mathfrak{g}}(e_5)\cdot(e_5\cdot e_5),\\
% &(e_5\cdot e_6)\cdot\phi_{\mathfrak{g}}(e_5)-\phi_{\mathfrak{g}}(e_5)\cdot(e_6\cdot e_5)=-\frac{1}{2}e_2=(e_6\cdot e_5)\cdot\phi_{\mathfrak{g}}(e_5)-\phi_{\mathfrak{g}}(e_6)\cdot(e_5\cdot e_5),\\
% &
 %(e_5\cdot e_6)\cdot\phi_{\mathfrak{g}}(e_6)-\phi_{\mathfrak{g}}(e_5)\cdot(e_6\cdot e_6)=-\frac{1}{2}e_1=(e_6\cdot e_5)\cdot\phi_{\mathfrak{g}}(e_6)-\phi_{\mathfrak{g}}(e_6)\cdot(e_5\cdot e_6),\\
% &
 %(e_6\cdot e_6)\cdot\phi_{\mathfrak{g}}(e_6)-\phi_{\mathfrak{g}}(e_6)\cdot(e_6\cdot e_6)=\frac{1}{2}e_1=(e_6\cdot e_6)\cdot\phi_{\mathfrak{g}}(e_6)-\phi_{\mathfrak{g}}(e_6)\cdot(e_6\cdot e_6).
 \end{align*}
 The above equations imply (\ref{L95}) holds. Therefore the holomorphic Norden Hom-Lie algebra 	$(\mathfrak{g^!}, [\cdot, \cdot], \phi_{\mathfrak{g^!}},J,\langle\cdot,\cdot\rangle)$ has  a Hom-Left-symmetric algebra structure.
\end{example}
%%%%%%%%%%%%%%%%%%%%%%%%%%%%%%%%%%%%%%%%%%%%%%%%%%%%%%%%%%%%%%%%%%%%%%%%%%%%%%%%%%%%%
 %%%%%%%%%%%%%%%%%%%%%%%%%%%%%%%%%%%%%%%%%%%%%%%%%%%%%%%%
%%%%%%%%%%%%%%%%%%%%%%%%%%%%%%%%%%%%%%%%%%%%%%%%%%%%

%*******************************************************************************************

%%%%%%%%%%%%%%%%%%%%%%%%%%%%%%%%%%%%%%%%%%
%%%%%%%%%%%%%%%%%%%%%%%%%%%%%%%%%%%%%%%

\end{document}